\documentclass[11pt]{article}
\usepackage[a4paper,margin=1in]{geometry}
\usepackage[T1]{fontenc}
\usepackage[utf8]{inputenc}
\usepackage{lmodern}
\usepackage{amsmath,amssymb,amsthm}
\usepackage{float}
\usepackage{enumitem}
\usepackage{hyperref}
\usepackage{microtype}
\usepackage{xcolor}
\usepackage{tikz}
\usepackage{standalone}
\usetikzlibrary{calc}

\definecolor{crossedge}{HTML}{B8C0CC}
\definecolor{circulantedge}{HTML}{718096}
\definecolor{matchingedge}{HTML}{E05A47}

\definecolor{circulantedge}{HTML}{718096}
\definecolor{matchingedge}{HTML}{E05A47}

\hypersetup{colorlinks=false}

\newcommand{\floor}[1]{\left\lfloor #1 \right\rfloor}
\newcommand{\ceil}[1]{\left\lceil #1 \right\rceil}
\newcommand{\cl}{\operatorname{cl}}

\numberwithin{equation}{section}
\newtheorem{theorem}{Theorem}[section]
\newtheorem{problem}[theorem]{Problem}
\newtheorem{lemma}[theorem]{Lemma}

\theoremstyle{plain}

\theoremstyle{definition}
\newtheorem{cons}[theorem]{Construction}
\newtheorem{remark}[theorem]{Remark}
\theoremstyle{plain}
\newtheorem{claim}{Claim}[section]

\title{On Erd\H{o}s Problem 767: Cycles with Chords}
\author{
Xiaozheng Chen\footnote{School of Mathematics and Statistics, Zhengzhou University, Zhengzhou 450000, P.R. China. Email: cxz@zzu.edu.cn. Supported by the National Natural Science Foundation of China (No. 12301457).}
, Bo Ning\footnote{College of Computer Science, Nankai University, Tianjin 300350, P.R. China. Email: bo.ning@nankai.edu.cn. Partially supported by the National Natural Science Foundation of China (No. 12371350) and the Fundamental Research Funds for the Central Universities, Nankai University (No. 63243151).}}
\date{}

\begin{document}
\maketitle

\begin{abstract}
For integers $k\ge 1$ and $n\ge k+2$, let $g_k(n)$ be the maximum number
of edges in an $n$-vertex graph containing no cycle with a vertex incident
with at least $k$ chords. Erd\H{o}s conjectured that
$g_k(n)=(k+1)(n-k-1)$ for $n\ge 2k+2$. Lewin found a counterexample. Bollob\'as later
conjectured that there exists a function $n(k)$ such that $g_k(n)=(k+1)(n-k-1)$ for all $n\ge n(k)$.
Jiang confirmed this by proving the formula for all $n\ge3k+3$
when $k\ge1$.

In this paper, we determine $g_k(n)$ completely.  For all $k\ge1$ and $n\ge k+2$, we prove
$$
g_k(n)=
\max\left\{
\left\lfloor\frac{(k+1)n}{2}\right\rfloor,
\max_{\substack{a\in\mathbb Z\\
    \left\lfloor {(k+1)}/{2}\right\rfloor+1\le a\le k+1}}
\left\{a(n-a)+\left\lfloor\frac{a(k+1-a)}{2}\right\rfloor\right\}
\right\}.
$$
For $k\ge2$, we prove
$g_k(n)=(k+1)(n-k-1)$ when
$n\ge \lceil(5k+1)/2\rceil$, and this threshold is sharp. 
Our proof builds on the method developed by Ma and the second author in \cite{MaNing2020}.
\end{abstract}



\section{Introduction}
All graphs in this paper are finite and simple. A \emph{chord} of a cycle
$C$ is an edge joining two nonconsecutive vertices of $C$. A cycle with a chord is
called a \emph{chorded cycle}. We study cycles containing a vertex incident with
many chords. 

For $k\ge1$ and $n\ge k+2$, let $g_k(n)$ denote the maximum number of
edges in an $n$-vertex graph containing no cycle with a vertex incident
with at least $k$ chords. 
The case $k=1$ follows from a classical theorem of
P\'osa~\cite[solution prob, 10.2, p.376]{Posa1961}: every graph of order $n\ge4$ with at
least $2n-3$ edges contains a chorded cycle.
The examples $K_3$ and $K_{2,n-2}$ show that both $n\ge 4$ and $2n-3$ are sharp.
For general $k$ and $n\ge 2k+2$, Erd\H{o}s~\cite{Erdos1969} conjectured that
$g_k(n)=(k+1)(n-k-1)$.
This conjecture was listed as Problem \#767 on Bloom's Erd\H{o}s Problems website~\cite{Bloom767}.
Lewin disproved this conjecture; see
\cite[p.~398, Problem~12]{Bollobas1978}. Bollob\'as then posed the following problem.

\begin{problem}[{\normalfont Bollob\'as~\cite[p.~398, Problem~13]{Bollobas1978}}]
\label{prob:Bollobas}
For every integer $k\ge 1$, does there exist an integer $n(k)$ such that
$g_k(n)=(k+1)(n-k-1)
$
for every $n\ge n(k)$?
\end{problem}

Jiang answered Problem~\ref{prob:Bollobas} affirmatively.

\begin{theorem}[Jiang~\cite{Jiang2004}]
\label{thm:Jiang}
Let $k\ge 1$ and $n\ge3k+3$. Then
$g_k(n)=(k+1)(n-k-1)$.
\end{theorem}

Jiang's theorem leaves the range $k+2\le n\le3k+2$ open. Our main result
determines $g_k(n)$ completely.

\begin{theorem}
\label{thm:main}
Let $k\ge1$ and $n\ge k+2$. Then
$$
g_k(n)=
\max\left\{
\left\lfloor\frac{(k+1)n}{2}\right\rfloor,
\max_{\substack{a\in\mathbb Z\\
    \left\lfloor {(k+1)}/{2}\right\rfloor+1\le a\le k+1}}
\left\{
a(n-a)+
\left\lfloor\frac{a(k+1-a)}{2}\right\rfloor
\right\}
\right\}.
$$
\end{theorem}

For $k\ge2$, Theorem~\ref{thm:main} implies that
$g_k(n)=(k+1)(n-k-1)$
for 
$n\ge \left\lceil{(5k+1)}/{2}\right\rceil$.
This threshold is sharp.
Indeed, when
$n=\lceil(5k+1)/2\rceil-1$, 
we can construct a graph with $(k+1)(n-k-1)+1$ edges.

The main theorem and its proof are closely related to the work of Ma and Ning on stability on circumference of graphs ~\cite{MaNing2020}. Stability results for Kopylov's theorem for graphs with bounded circumference were studied by F\"{u}redi, Kostochka, and Verstra\"ete~\cite{FKV2016}, and were later completed by F\"{u}redi, Kostochka, Luo, and Verstra\"ete~\cite{FKLV2018}.
Ma and Ning~\cite{MaNing2020} established a stability version of a conjecture proposed by Woodall in 1976~(see \cite{Woodall1976}).
They also discovered that Woodall's conjecture is related to a theorem on long cycles proved by Bondy in 1971~\cite{Bondy1971}.
As a tool for their proof, they obtained a stability version of Bondy's theorem.
Motivated by the idea of Jiang \cite{Jiang2004} using
Bondy’s theorem to prove that Problem~\ref{prob:Bollobas} holds for $n\ge 3k+3$,
in this paper, we use Ma and Ning’s stability result~\cite{MaNing2020} to handle the remaining range of $n$.

The study of chords in cycles is also related to several classical extremal and structural problems on cycles. Thomassen~\cite{Thomassen1989} conjectured that every longest cycle in a $3$-connected graph has a chord. Thomassen~\cite{Thomassen2018} first verified the conjecture for several classes of graphs and recently proved that every $3$-connected graph contains a longest cycle with a chord~\cite{Thomassen2026}. Another related problem concerns nested cycles. Bollob\'{a}s~\cite{2Bollobas1978} introduced this problem. Chen, Erd\H{o}s, and Staton~\cite{Chen1996} proved that, for every
fixed positive integer $k$, there exists a constant $c_k$ such that
every $n$-vertex graph with at least $c_k n$ edges contains $k$
edge-disjoint nested cycles. Gil Fern\'andez, Kim, Kim, and Liu~\cite{FKKL2022} proved that there
exists an absolute constant $C>0$ such that every $n$-vertex graph
with at least $Cn$ edges contains two edge-disjoint nested cycles
without crossings. 
Chakraborti, Janzer, Methuku, and Montgomery~\cite{CJMM2025} proved that, for some absolute constant $t$, $O_k(n(\log n)^t)$ edges are sufficient to guarantee $k$ pairwise edge-disjoint cycles with the same vertex set. For further results on chorded and nested cycles, see~\cite{Balister2018,HST2012,Voss1982}.

The paper is organized as follows.
We begin with a brief proof sketch in Section~\ref{sec:sketch}.
In Section~\ref{sec:constructions},
we present the extremal constructions and establish several estimates for the upper bound. 
In Section~\ref{sec:tools},
we collect 
the closure, switching, and
stability tools.
Finally, we prove 
Theorem~\ref{thm:main} in Section~\ref{sec:proof}.

\section{Notation and proof sketch}\label{sec:sketch}

For $X\subseteq V(G)$, let $G[X]$ denote the subgraph induced by $X$ and let
$G-X:=G[V(G)\setminus X]$. For $v\in V(G)$, let $G-v:=G-\{v\}$.
If $H$ is a subgraph of $G$, let $G-H:=G-V(H)$.
For disjoint sets $X,Y\subseteq V(G)$, let $E_G(X,Y)$ be the set of edges
with one end in $X$ and the other in $Y$, and set
$e_G(X,Y):=\lvert E_G(X,Y)\rvert$. 
For two disjoint subgraphs $H_1$ and $H_2$ of $G$, let
$e_G(H_1,H_2):=e_G(V(H_1),V(H_2))$.
Let $e(G):=\lvert E(G)\rvert$.

For each \(v\in V(G)\), let \(N_G(v)\) denote the neighborhood of \(v\),
set \(d_G(v):=\lvert N_G(v)\rvert\), and let
\(\delta(G):=\min_{v\in V(G)}d_G(v)\). 
For \(X\subseteq V(G)\), let 
$N_G(X):=\{v\in V(G)\setminus X : N_G(v)\cap X\neq\varnothing\}.$
For a subgraph $H$ of $G$ and a cycle $C$ in $G$, set
$N_C(H):=N_G(V(H))\cap V(C).$
For $x\in V(G)$, set
$N_H(x):=N_G(x)\cap V(H).$

The length of a path or cycle is its number of edges.
For a path $P$
and vertices $u,v\in V(P)$,
let $uPv$ be the subpath of $P$ from $u$ to $v$, including both endpoints.
For a subgraph $H$ of $G$,
an
$(x,H,y)$-path is an $x$--$y$ path whose internal vertices lie in
$V(H)$. An $(x,Z,y)$-path is defined similarly for a vertex set $Z$.
Let $\omega(G)$ be the clique number and $c(G)$ for the
circumference of $G$.
When we say that $C$ is oriented, we mean that it is oriented clockwise or counterclockwise.
If $C$ is oriented, then $v^+$ and $v^-$ denote the successor and
predecessor of $v$ on $C$.
If $C$ is oriented and $u,v\in V(C)$, then $uCv$ denotes the directed $u$--$v$ segment of $C$.
A cycle $C$ in $G$ is \emph{locally maximal}
if there is no cycle $C'$ with $\lvert E(C')\rvert>\lvert E(C)\rvert$ that uses at most two edges of
$E_G(V(C),V(G)\setminus V(C))$. Every longest cycle is locally maximal.
A graph is Hamiltonian-connected if every two distinct vertices are joined by a Hamilton path.

For an integer $r$, the $r$-closure of $G$ is obtained by repeatedly
joining two nonadjacent vertices whose degree sum is at least $r$. The
resulting graph is unique~\cite{Bondy1971,Chvatal1972}. Thus $G$ is
$r$-closed if $d_G(x)+d_G(y)<r$ for every pair of nonadjacent vertices
$x,y$. If $C$ is a cycle of length $c$, the \emph{$C$-closure} of $G$,
denoted by $\cl_C(G)$, is obtained by replacing $G[V(C)]$ with its
$(c+1)$-closure and leaving all other adjacencies unchanged.

\medskip
\noindent
{\bf Proof sketch.}
The case $k=1$ follows from P\'osa's theorem, so we may assume that
$k\ge2$. The lower bound follows from the two constructions in
Section~\ref{sec:constructions}: a nearly $(k+1)$-regular construction
and a join-type construction. For the upper bound, we first establish it for
$k+2\le n\le \ceil{(5k+1)/2}$
and then extend the bound to all $n$ by induction.

First, suppose that
$k+2\le n\le \ceil{(5k+1)/2}$.
We argue by contradiction and choose a minimal counterexample $G$.
Then $G$ is $2$-connected and non-Hamiltonian. Let $C$ be a longest
cycle of $G$, and set $c=\lvert E(C)\rvert$.
Applying Lemma~\ref{lem:Ma-Ning-cycle-dichotomy} to $G$ and $C$,
we obtain two cases for the subgraph induced by $V(C)$ in the
$C$-closure of $G$: either deleting $s-1$ vertices of degree at most
$s$ leaves a clique, or there are $\floor{c/2}-1$ vertices of degree
at most $\floor{c/2}$.
These two cases give bounds on
$e(G[V(C)])$ and on the number of edges with at least one endpoint
outside $C$.

Next, we first treat the case $c\le9$ separately, and then suppose that $c\ge10$. If
$e(G-C)+e_G(G-C,C)
   \le \bigl(\floor{c/2}-1\bigr)(n-c)$,
then combining this inequality with the degree bounds on $C$ gives
the bound in Theorem~\ref{thm:main}. Otherwise, the Ma--Ning stability
theorem~\cite{MaNing2020} (see Theorem~\ref{thm:MN-exterior-stability}) implies that $G$ belongs to one of the exceptional families or $G\subseteq W_{n,\floor{c/2},c}$,
where the definition of $W_{n,\floor{c/2},c}$ is postponed to 
Section~\ref{sec:tools}. 
We also show that
if $G$ belongs to one of the exceptional families, then 
$G\subseteq W_{n,\floor{c/2},c}$.
Finally,
the structure of
$W_{n,\floor{c/2},c}$ yields the bound in Theorem~\ref{thm:main}.

For $n\ge \ceil{(5k+1)/2}$, the upper bound follows by induction on $n$,
using an endpoint $v$ of a longest path with $d_G(v)\le k+1$.

\section{Extremal constructions and arithmetic}\label{sec:constructions}

For $\ell\ge2$, an \textit{$\ell$-path-fan} in $G$ consists of a vertex
$x$ and a path $P$ in $G-x$ such that
$\lvert N_G(x)\cap V(P)\rvert\ge\ell$.
We call $x$ the center of the path-fan.

\begin{lemma}\label{lem:path-fan-equivalence}
For $k\ge1$, a graph contains a $(k+2)$-path-fan if and only if it contains a
cycle with a vertex incident with at least $k$ chords.
\end{lemma}

\begin{proof}
First, suppose that a vertex $x$ is incident with at least $k$ chords of
a cycle $C$. Then $x$ has at least $k+2$ neighbors on the path $C-x$.
Hence $x$ and $C-x$ form a $(k+2)$-path-fan.

Conversely, suppose that $x$ and a path $P$ in $G-x$ form a
$(k+2)$-path-fan. Let $uPw$ be the minimal subpath of $P$ containing
$N_G(x)\cap V(P)$. In particular, both $u$ and $w$ are adjacent to $x$. The cycle $xuPwx$ contains all neighbors of $x$ on $uPw$. Thus, $x$ is incident with at
least $k$ chords of this cycle.
\end{proof}

For every integer $k\ge 2$, set
$$
I_k:=\left\{a\in\mathbb Z:
\left\lfloor\frac{k+1}{2}\right\rfloor+1\le a\le k+1\right\}.
$$
For an integer $n\ge 1$, define
$$
R_k(n):=\left\lfloor\frac{(k+1)n}{2}\right\rfloor,
\qquad
J_k(n):=\max_{a\in I_k}a(n-a)+\left\lfloor \frac{a(k+1-a)}{2}\right\rfloor.$$
Finally, set
$$
F_k(n):=
\begin{cases}
\binom n2, & 1\le n\le k+1,\\
\max\{R_k(n),J_k(n)\}, & n\ge k+2.
\end{cases}
$$
\subsection{The lower-bound constructions}
We give two families of graphs containing no $(k+2)$-path-fan: a nearly regular construction with $R_k(n)$ edges and a join-type construction with $a(n-a)+\left\lfloor {a(k+1-a)}/{2}\right\rfloor$ edges for each $a\in I_k$.

\begin{cons}\normalfont
Take the cyclic group $\mathbb{Z}_n$ where $n\ge k+2$.
Define edge set
$M_1=\{\{i,i+j\}: i\in \mathbb{Z}_n,\ 1\le j\le \floor{{(k+1)}/{2}}\}$.
Moreover, if $k+1$ is odd,
$M_2=\{\{i,i+\floor{n/2}\}: 0\le i\le \floor{n/2}-1\}$,
where all additions are taken in $\mathbb{Z}_n$.
Then define the graph $G$ as follows:
\begin{itemize}[nosep]
    \item if $k+1$ is even, let
    $V(G)=\mathbb{Z}_n$, $E(G)=M_1$;
    \item if $k+1$ is odd, let
    $V(G)=\mathbb{Z}_n$, $E(G)=M_1\cup M_2$.
\end{itemize}
If $k+1$ is even, then $G$ is $(k+1)$-regular, and 
$e(G)={(k+1)n}/{2}=R_k(n)$.
If $k+1$ is odd, then $(\mathbb{Z}_n,M_1)$ is $k$-regular. Since $n\ge k+2$, we have
$\floor{n/2}>\floor{(k+1)/2}$.
Then $M_2$ is disjoint from $M_1$ and forms a
matching of size $\floor{n/2}$. Thus the maximum degree of $G$ is at most $k+1$ and
$e(G)={kn}/{2}+\floor{n/2}
=\floor{{(k+1)n}/{2}}=R_k(n)$.
Since the center of a $(k+2)$-path-fan has at least $k+2$ distinct neighbors,
every graph with maximum degree at most $k+1$ contains no $(k+2)$-path-fan; see Figure~\ref{fig:nearly-regular-construction}.
\end{cons}

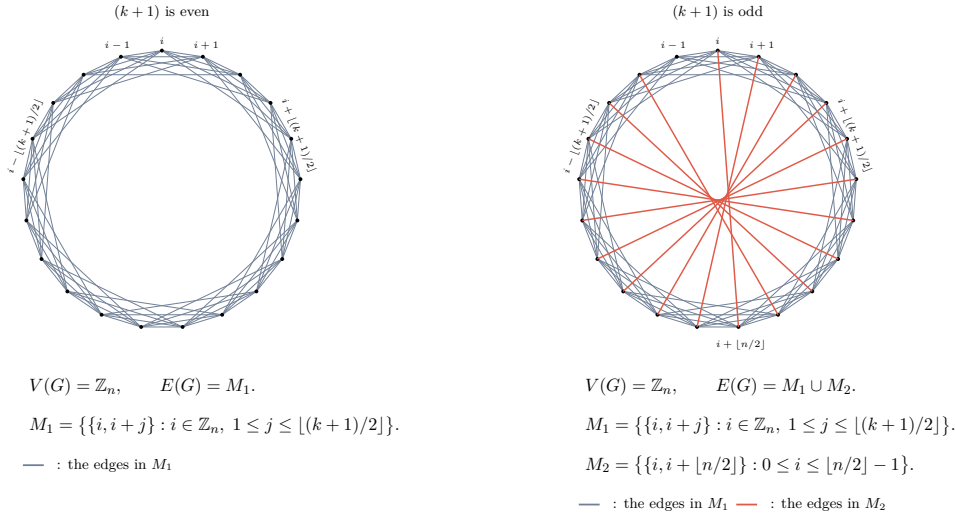
\begin{figure}[H]
    \centering
\begin{tikzpicture}[
    circulant edge/.style={
        draw=circulantedge,
        line width=0.45pt
    },
    matching edge/.style={
        draw=matchingedge,
        line width=0.9pt
    },
    vertex/.style={
        circle,
        fill=black,
        inner sep=0pt,
        minimum size=2.4pt
    }
]

\newcommand{\drawcirculant}[1]{%
    \foreach \i in {0,...,\nminusone} {
        \pgfmathsetmacro{\ang}{90-360*\i/\n}
        \coordinate (#1v\i) at (\ang:3cm);
    }

    \foreach \i in {0,...,\nminusone} {
        \foreach \j in {1,...,\m} {
            \pgfmathtruncatemacro{\t}{mod(\i+\j,\n)}
            \draw[circulant edge]
                (#1v\i)--(#1v\t);
        }
    }

    \foreach \i in {0,...,\nminusone} {
        \node[vertex] at (#1v\i) {};
    }

    \node at (90:3.22cm)
        {\tiny $i$};

    \pgfmathsetmacro{\plusoneang}{90-360/\n}
    \node at (\plusoneang:3.25cm)
        {\tiny $i+1$};

    \pgfmathsetmacro{\pluslastang}{90-360*\m/\n}
    \pgfmathsetmacro{\pluslabelrotation}{\pluslastang-90}
    \node[
        rotate=\pluslabelrotation
    ]
        at (\pluslastang:3.17cm)
        {\tiny $i+\floor{(k+1)/2}$};

    \pgfmathsetmacro{\minusoneang}{90+360/\n}
    \node at (\minusoneang:3.25cm)
        {\tiny $i-1$};

    \pgfmathsetmacro{\minuslastang}{90+360*\m/\n}
    \pgfmathsetmacro{\minuslabelrotation}{\minuslastang-90}
    \node[
        rotate=\minuslabelrotation
    ]
        at (\minuslastang:3.17cm)
        {\tiny $i-\floor{(k+1)/2}$};
}

\newcommand{\drawmatching}[1]{%
    \foreach \i in {0,...,\matchlast} {
        \pgfmathtruncatemacro{\t}{mod(\i+\half,\n)}
        \draw[matching edge]
            (#1v\i)--(#1v\t);
    }

    \pgfmathsetmacro{\matchang}{90-360*\half/\n}
    \node at (\matchang:3.38cm)
        {\tiny $i+\floor{n/2}$};
}

\begin{scope}[xshift=6cm]

    \def\n{21}
    \def\k{8}

    \pgfmathtruncatemacro{\nminusone}{\n-1}
    \pgfmathtruncatemacro{\m}{floor((\k+1)/2)}
    \pgfmathtruncatemacro{\half}{floor(\n/2)}
    \pgfmathtruncatemacro{\matchlast}{\half-1}

    \drawcirculant{L}
    \drawmatching{L}

    \node at (0,3.82cm)
        {\footnotesize
         $(k+1)\text{ is odd}$};

    \node[anchor=west]
        at (-3cm,-4.25cm) {%
        $\displaystyle
        V(G)=\mathbb{Z}_n,
        \qquad
        E(G)=M_1\cup M_2.$
    };

    \node[anchor=west]
        at (-3cm,-5.10cm) {%
        $\displaystyle
        M_1=
        \bigl\{
            \{i,i+j\}:
            i\in\mathbb{Z}_n,\;
            1\le j\le\floor{(k+1)/2}
        \bigr\}.$
    };

    \node[anchor=west]
        at (-3cm,-5.95cm) {%
        $\displaystyle
        M_2=
        \bigl\{
            \{i,i+\floor{n/2}\}:
            0\le i\le\floor{n/2}-1
        \bigr\}.$
    };

    \draw[
        circulant edge,
        line width=1.1pt
    ]
        (-3cm,-6.80cm)
        --
        (-2.55cm,-6.80cm);

    \node[anchor=west]
        at (-2.40cm,-6.80cm) {%
        \footnotesize
        : the edges in $M_1$
    };

    \draw[
        matching edge,
        line width=1.1pt
    ]
        (0.40cm,-6.80cm)
        --
        (0.85cm,-6.80cm);

    \node[anchor=west]
        at (1.00cm,-6.80cm) {%
        \footnotesize
        : the edges in $M_2$
    };

\end{scope}

\begin{scope}[xshift=-6cm]

    \def\n{21}
    \def\k{7}

    \pgfmathtruncatemacro{\nminusone}{\n-1}
    \pgfmathtruncatemacro{\m}{floor((\k+1)/2)}

    \drawcirculant{R}

    \node at (0,3.82cm)
        {\footnotesize
         $(k+1)\text{ is even}$};

    \node[anchor=west]
        at (-3cm,-4.25cm) {%
        $\displaystyle
        V(G)=\mathbb{Z}_n,
        \qquad
        E(G)=M_1.$
    };

    \node[anchor=west]
        at (-3cm,-5.10cm) {%
        $\displaystyle
        M_1=
        \bigl\{
            \{i,i+j\}:
            i\in\mathbb{Z}_n,\;
            1\le j\le\floor{(k+1)/2}
        \bigr\}.$
    };

    \draw[
        circulant edge,
        line width=1.1pt
    ]
        (-3cm,-5.95cm)
        --
        (-2.55cm,-5.95cm);

    \node[anchor=west]
        at (-2.40cm,-5.95cm) {%
        \footnotesize
        : the edges in $M_1$
    };

\end{scope}


\end{tikzpicture}
    \caption{The nearly regular construction.}
    \label{fig:nearly-regular-construction}
\end{figure}

\begin{cons}\label{cons:split-construction}\normalfont
Let $a\in I_k$. Partition the vertex set into $X$ and $Y$ with
$\lvert X\rvert=a$ and $\lvert Y\rvert=n-a$, where $Y$ is an independent set.
Identify $X$ with the cyclic group $\mathbb{Z}_a$.
Define two edge sets
$M_0=\{\{x,y\}:x\in X,\ y\in Y\}$,
and
$M_1=\{\{i,i+j\}:i\in\mathbb{Z}_a,\ 
1\le j\le \lfloor (k+1-a)/2\rfloor\}$.
Moreover, if $k+1-a$ is odd, define
$M_2=\{\{i,i+\lfloor a/2\rfloor\}:
0\le i\le \lfloor a/2\rfloor-1\}$,
where all additions are taken in $\mathbb{Z}_a$.
Then we define the graph $G$ as follows:
\begin{itemize}[nosep]
\item if $k+1-a$ is even, let
$V(G)=X\cup Y$, $E(G)=M_0\cup M_1$;
\item if $k+1-a$ is odd, let
$V(G)=X\cup Y$, $E(G)=M_0\cup M_1\cup M_2$.
\end{itemize}
The resulting graph $G$ has
$e(G)=a(n-a)+\left\lfloor{a(k+1-a)}/{2}\right\rfloor$
edges and contains no $(k+2)$-path-fan.
Indeed, if the center of a $(k+2)$-path-fan is a vertex
$y\in Y$, then $N_G(y)=X$.
Hence $d_G(y)=a\le k+1$.
If the center is a vertex $x\in X$, let $P$ be any path in
$G-x$. Since $Y$ is independent, $P$ contains at most $a$
vertices of $Y$. Moreover, $x$ has at most $k+1-a$ neighbors
in $X$. Therefore,
$\lvert N_G(x)\cap V(P)\rvert
\le a+(k+1-a)=k+1$.
Thus $G$ contains no $(k+2)$-path-fan; see
Figure~\ref{fig:split-construction}.
\end{cons}
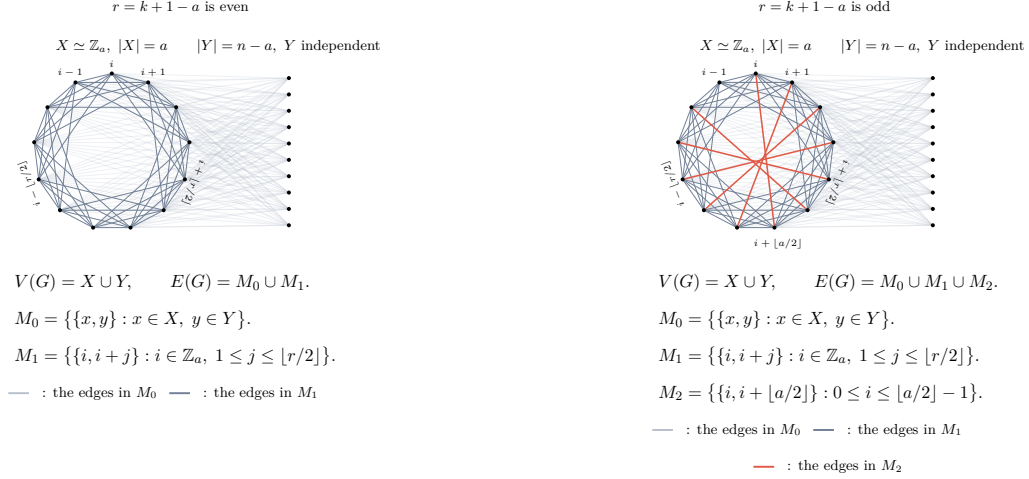
\begin{figure}[H]
    \centering
    \begin{tikzpicture}[
    cross edge/.style={
        draw=crossedge,
        line width=0.32pt,
        opacity=0.24
    },
    circulant edge/.style={
        draw=circulantedge,
        line width=0.55pt
    },
    matching edge/.style={
        draw=matchingedge,
        line width=0.95pt
    },
    vertex/.style={
        circle,
        fill=black,
        inner sep=0pt,
        minimum size=2.5pt
    }
]

\newcommand{\drawsplitbase}[1]{%
    \coordinate (#1center) at (-1.5cm,0);

    \foreach \i in {0,...,\aminusone} {
        \pgfmathsetmacro{\ang}{90-360*\i/\a}
        \coordinate (#1x\i)
            at ($(#1center)+(\ang:1.70cm)$);
    }

    \foreach \q in {0,...,\yminusone} {
        \pgfmathsetmacro{\yy}{1.60-3.20*\q/\yminusone}
        \coordinate (#1y\q)
            at (2.35cm,\yy cm);
    }

    \foreach \i in {0,...,\aminusone} {
        \foreach \q in {0,...,\yminusone} {
            \draw[cross edge]
                (#1x\i)--(#1y\q);
        }
    }

    \foreach \i in {0,...,\aminusone} {
        \foreach \j in {1,...,\m} {
            \pgfmathtruncatemacro{\t}{mod(\i+\j,\a)}
            \draw[circulant edge]
                (#1x\i)--(#1x\t);
        }
    }
}

\newcommand{\drawsplitmatching}[1]{%
    \foreach \i in {0,...,\matchlast} {
        \pgfmathtruncatemacro{\t}{mod(\i+\half,\a)}
        \draw[matching edge]
            (#1x\i)--(#1x\t);
    }
}

\newcommand{\drawsplitnodes}[1]{%
    \foreach \i in {0,...,\aminusone} {
        \node[vertex] at (#1x\i) {};
    }

    \foreach \q in {0,...,\yminusone} {
        \node[vertex] at (#1y\q) {};
    }

    \node at ($(#1center)+(90:1.91cm)$)
        {\tiny $i$};

    \pgfmathsetmacro{\plusoneang}{90-360/\a}
    \node at ($(#1center)+(\plusoneang:1.94cm)$)
        {\tiny $i+1$};

    \pgfmathsetmacro{\pluslastang}{90-360*\m/\a}
    \pgfmathsetmacro{\plusrotation}{\pluslastang-90}
    \node[rotate=\plusrotation]
        at ($(#1center)+(\pluslastang:1.91cm)$)
        {\tiny $i+\left\lfloor r/2\right\rfloor$};

    \pgfmathsetmacro{\minusoneang}{90+360/\a}
    \node at ($(#1center)+(\minusoneang:1.94cm)$)
        {\tiny $i-1$};

    \pgfmathsetmacro{\minuslastang}{90+360*\m/\a}
    \pgfmathsetmacro{\minusrotation}{\minuslastang-90}
    \node[rotate=\minusrotation]
        at ($(#1center)+(\minuslastang:1.91cm)$)
        {\tiny $i-\left\lfloor r/2\right\rfloor$};

    \node at (-1.5cm,2.28cm)
        {\footnotesize
         $X\simeq\mathbb{Z}_a,\ |X|=a$};

    \node at (2.35cm,2.28cm)
        {\footnotesize
         $|Y|=n-a,\ Y\text{ independent}$};
}

\begin{scope}[xshift=-7cm]

    \def\a{13}
    \def\r{8}
    \def\ycount{10}

    \pgfmathtruncatemacro{\aminusone}{\a-1}
    \pgfmathtruncatemacro{\yminusone}{\ycount-1}
    \pgfmathtruncatemacro{\m}{floor(\r/2)}

    \drawsplitbase{L}
    \drawsplitnodes{L}

    \node at (0,3.15cm)
        {\footnotesize
         $r=k+1-a\text{ is even}$};

    \node[anchor=west]
        at (-3.75cm,-2.85cm) {%
        $\displaystyle
        V(G)=X\mathbin{\cup}Y,
        \qquad
        E(G)=M_0\cup M_1.$
    };

    \node[anchor=west]
        at (-3.75cm,-3.65cm) {%
        $\displaystyle
        M_0=
        \bigl\{
            \{x,y\}:
            x\in X,\;
            y\in Y
        \bigr\}.$
    };

    \node[anchor=west]
        at (-3.75cm,-4.45cm) {%
        $\displaystyle
        M_1=
        \bigl\{
            \{i,i+j\}:
            i\in\mathbb{Z}_a,\;
            1\le j\le
            \left\lfloor r/2\right\rfloor
        \bigr\}.$
    };

    \draw[
        cross edge,
        opacity=1,
        line width=1.2pt
    ]
        (-3.75cm,-5.25cm)
        --
        (-3.30cm,-5.25cm);

    \node[anchor=west]
        at (-3.15cm,-5.25cm) {%
        \footnotesize
        : the edges in $M_0$
    };

    \draw[
        circulant edge,
        line width=1.2pt
    ]
        (-0.25cm,-5.25cm)
        --
        (0.20cm,-5.25cm);

    \node[anchor=west]
        at (0.35cm,-5.25cm) {%
        \footnotesize
        : the edges in $M_1$
    };

\end{scope}

\begin{scope}[xshift=7cm]

    \def\a{13}
    \def\r{9}
    \def\ycount{10}

    \pgfmathtruncatemacro{\aminusone}{\a-1}
    \pgfmathtruncatemacro{\yminusone}{\ycount-1}
    \pgfmathtruncatemacro{\m}{floor(\r/2)}
    \pgfmathtruncatemacro{\half}{floor(\a/2)}
    \pgfmathtruncatemacro{\matchlast}{\half-1}

    \drawsplitbase{R}
    \drawsplitmatching{R}
    \drawsplitnodes{R}

    \pgfmathsetmacro{\matchang}{90-360*\half/\a}

    \node
        at ($(Rcenter)+(\matchang:2.05cm)$)
        {\tiny
         $i+\left\lfloor a/2\right\rfloor$};

    \node at (0,3.15cm)
        {\footnotesize
         $r=k+1-a\text{ is odd}$};

    \node[anchor=west]
        at (-3.75cm,-2.85cm) {%
        $\displaystyle
        V(G)=X\mathbin{\cup}Y,
        \qquad
        E(G)=M_0\cup M_1\cup M_2.$
    };

    \node[anchor=west]
        at (-3.75cm,-3.65cm) {%
        $\displaystyle
        M_0=
        \bigl\{
            \{x,y\}:
            x\in X,\;
            y\in Y
        \bigr\}.$
    };

    \node[anchor=west]
        at (-3.75cm,-4.45cm) {%
        $\displaystyle
        M_1=
        \bigl\{
            \{i,i+j\}:
            i\in\mathbb{Z}_a,\;
            1\le j\le
            \left\lfloor r/2\right\rfloor
        \bigr\}.$
    };

    \node[anchor=west]
        at (-3.75cm,-5.25cm) {%
        $\displaystyle
        M_2=
        \bigl\{
            \{i,i+\left\lfloor a/2\right\rfloor\}:
            0\le i\le
            \left\lfloor a/2\right\rfloor-1
        \bigr\}.$
    };

    \draw[
        cross edge,
        opacity=1,
        line width=1.2pt
    ]
        (-3.75cm,-6.05cm)
        --
        (-3.30cm,-6.05cm);

    \node[anchor=west]
        at (-3.15cm,-6.05cm) {%
        \footnotesize
        : the edges in $M_0$
    };

    \draw[
        circulant edge,
        line width=1.2pt
    ]
        (-0.25cm,-6.05cm)
        --
        (0.20cm,-6.05cm);

    \node[anchor=west]
        at (0.35cm,-6.05cm) {%
        \footnotesize
        : the edges in $M_1$
    };

    \draw[
        matching edge,
        line width=1.2pt
    ]
        (-1.55cm,-6.85cm)
        --
        (-1.10cm,-6.85cm);

    \node[anchor=west]
        at (-0.95cm,-6.85cm) {%
        \footnotesize
        : the edges in $M_2$
    };

\end{scope}

\end{tikzpicture}
    \caption{The join-type construction.}
    \label{fig:split-construction}
\end{figure}

The two constructions show $g_k(n)\ge \max\{R_k(n),J_k(n)\}=F_k(n)$ for all $n\ge k+2$. Moreover, Construction~\ref{cons:split-construction} shows that the threshold $n=\ceil{(5k+1)/2}$ is sharp.

\begin{remark}
Let $k\ge2$ and $n=\ceil{(5k+1)/2}-1$.
Construction~\ref{cons:split-construction} with $a=k$ gives a graph containing no
$(k+2)$-path-fan and having $k(n-k)+\floor{k/2}$ edges, where
$k(n-k)+\floor{k/2}-(k+1)(n-k-1)
=1.$
Thus the threshold $n=\ceil{(5k+1)/2}$ is best possible.
\end{remark}

\subsection{Estimates for the upper bound}
We next prove several lemmas. 
These lemmas will be used in the calculations in Section~\ref{sec:proof} and in the induction. Readers who are mainly interested in the proof of Theorem \ref{thm:main} can skip these lemmas and go directly to the main proof in Section~\ref{sec:proof} to understand the overall structure.
They can also 
use a computer program to check the calculations independently.

Recall the functions $R_k,J_k,F_k$ defined in the beginning of Section \ref{sec:constructions}.
The first lemma determines when $J_k(n)$ exceeds $R_k(n)$ and will be used to derive the piecewise formula for $F_k(n)$.

\begin{lemma}\label{lem:linear-bound}
Let $k\ge 2$ and let $N_k\ge k+2$ be the smallest integer such that
$J_k(N_k)>R_k(N_k)$.
Then $J_k(m)>R_k(m)$ for all $m\ge N_k$.
Moreover,
\begin{enumerate}[label=\textup{(\roman*)},nosep]
    \item $\floor{\left(1+\frac{\sqrt{3}}{2}\right)(k+1)}+1
\le N_k\le
\ceil{\left(1+\frac{\sqrt{3}}{2}\right)(k+1)}+1$.
\item if $J_k(n)\le R_k(n)$ for some $n\ge k+2$, then $n\le \bigl\lceil\bigl(1+\frac{\sqrt3}{2}\bigr)(k+1)\bigr\rceil \le 2k+2$.
\end{enumerate}
\end{lemma}
\begin{proof}
For $k=2$, we have
$R_2(n)=\lfloor3n/2\rfloor$ and
$J_2(n)=\max\{2n-3,3n-9\}$. Thus, $N_2=7$, and both assertions follow.
We may therefore assume that $k\ge3$.
The integer $N_k$ exists because
$(k+1)(n-k-1)>R_k(n)$ for all sufficiently large $n$.

Let $a\in I_k$ be such that
$J_k(N_k)=a(N_k-a)+\left\lfloor\frac{a(k+1-a)}{2}\right\rfloor$.
For any $m>N_k$,
we have
\begin{align*}
J_{k}(m) &\ge a(m-a)+\left\lfloor\frac{a(k+1-a)}{2}\right\rfloor
          = a(N_k-a)+\left\lfloor\frac{a(k+1-a)}{2}\right\rfloor + a(m-N_k) \\
          &= J_{k}(N_k)+a(m-N_k),\\[4pt]
R_k(m) &=\left\lfloor\frac{(k+1)m}{2}\right\rfloor
        \le \left\lfloor\frac{(k+1)N_k}{2}\right\rfloor + a(m-N_k)
        = R_k(N_k)+a(m-N_k).
\end{align*}
Since $J_k(N_k)>R_k(N_k)$, we obtain
$J_k(m) > R_k(m)$ for all $m\ge N_k$.
We prove the remaining assertion by constructing a function $\Phi(x)$ 
and dividing $n$ into two ranges according to the minimum value of $\Phi(x)$.

(i) Fix $a\in I_k$.
We first compare $a(n-a)+\frac{a(k+1-a)}{2}$ and $\frac{(k+1)n}{2}$.
Let
$x:=2a-k-1$
and define
$\Phi(x):=
\frac{(k+1)^2}{4x}+(k+1)+\frac{3x}{4}.$
Then
\begin{equation}\label{eq:crossover-factorization}
a(n-a)+\frac{a(k+1-a)}{2}-\frac{(k+1)n}{2}
=
\frac{x}{2}\bigl(n-\Phi(x)\bigr).
\end{equation}
The function $\Phi$ attains its minimum at
$x_0=\frac{k+1}{\sqrt3}$,
where
$\Phi(x_0)=\left(1+\frac{\sqrt3}{2}\right)(k+1)$.

\noindent{\bf Case~1:} $n\le\bigl\lfloor\Phi(x_0)\bigr\rfloor$.

For every $a\in I_k$, we have
$n\le\Phi(x_0)\le\Phi(x)$. From~\eqref{eq:crossover-factorization},
$a(n-a)+\left\lfloor\frac{a(k+1-a)}{2}\right\rfloor\le \left\lfloor\frac{(k+1)n}{2}\right\rfloor$ for all $a\in I_k$. Hence
$J_k(n)\le R_k(n)$ when $k+2\le n\le\bigl\lfloor\Phi(x_0)\bigr\rfloor$.

\noindent{\bf Case~2: } $n\ge \bigl\lceil\Phi(x_0)\bigr\rceil+1$.

Set $N=\bigl\lceil\Phi(x_0)\bigr\rceil+1$.
Choose an integer $x$ such that $x\equiv k+1\pmod2$ and $|x-x_0|\le1$.
Such a choice is possible because $k+1\ge4$.
Define $a=\frac{k+1+x}{2}$. Then $a$ is an integer and
$\bigl\lfloor\frac{k+1}{2}\bigr\rfloor+1\le a\le k+1$.
Hence $a\in I_k$.
Since $|x-x_0|<1$ and $x\ge2$, we have
$0\le\Phi(x)-\Phi(x_0)=\frac{3}{4}\frac{(x-x_0)^2}{x}
<\frac{3}{4x}\le\frac{3}{8}$.
Consequently,
$N-\Phi(x)=\bigl(N-\Phi(x_0)\bigr)-\bigl(\Phi(x)-\Phi(x_0)\bigr)
>1-\frac{3}{8}=\frac{5}{8}$.
From~\eqref{eq:crossover-factorization},
we have
$a(N-a)+{\frac{a(k+1-a)}{2}}-\frac{(k+1)N}{2}
= \frac{x}{2}\bigl(N-\Phi(x)\bigr)
> \frac12$.
Then,
$$
J_{k}(N) \ge a(N-a)+\left\lfloor\frac{a(k+1-a)}{2}\right\rfloor
\ge a(N-a)+{\frac{a(k+1-a)}{2}}-\frac12
> \frac{(k+1)N}{2} \ge R_k(N).
$$
Therefore $J_k(N)> R_k(N)$.

Since $J_k(n)\le R_k(n)$ for all $n\le\bigl\lfloor\Phi(x_0)\bigr\rfloor$ by Case~1,
we have $N_k\ge\bigl\lfloor\Phi(x_0)\bigr\rfloor+1$.
Since $J_k(N)>R_k(N)$ with $N=\lceil\Phi(x_0)\rceil+1$,
we have $J_k(m)>R_k(m)$ for every $m\ge N$. Therefore, $N_k\le N=\lceil\Phi(x_0)\rceil+1$.
Combining the two bounds, we obtain
$$
\Bigl\lfloor\Bigl(1+\frac{\sqrt3}{2}\Bigr)(k+1)\Bigr\rfloor+1
\;\le\; N_k \;\le\;
\Bigl\lceil\Bigl(1+\frac{\sqrt3}{2}\Bigr)(k+1)\Bigr\rceil+1.
$$

(ii) If  $n\ge k+2$ and $J_k(n)\le R_k(n)$, then $n<N_k$. Consequently, we have
$n\le N_k-1\le\bigl\lceil\Phi(x_0)\bigr\rceil$.
Since $\Phi(x_0)<2(k+1)$,
$\bigl\lceil\Phi(x_0)\bigr\rceil\le2k+2$.
This completes the proof.
\renewcommand{\qedsymbol}{$\blacksquare$}
\end{proof}

The following lemma shows that $J_k(n)$ is a linear function of $n$ when $n\ge \ceil{{(5k+1)}/{2}}$. 

\begin{lemma}\label{lem:join-term-linear-range}
Let $k\ge 2$. If $n\ge \ceil{{(5k+1)}/{2}}$, then
$J_k(n)=(k+1)(n-k-1)$.
\end{lemma}

\begin{proof}
It suffices to show that
$a(n-a)+\floor{\frac{a(k+1-a)}{2}}\le (k+1)(n-k-1)$ for every $a\in I_k$.
For $a=k+1$, equality holds. Let
$a=k+1-d$, where $d\in\mathbb Z$ and $1\le d<\frac{k+1}{2}$.
Then we have 
\begin{equation}\label{eq:d}
    \begin{aligned}
& (k+1)(n-k-1)
 -\left((k+1-d)(n-k-1+d)
      +\floor{\frac{d(k+1-d)}{2}}\right)  \\
=& d(n-2k-2+d)-\floor{\frac{d(k+1-d)}{2}} .
\end{aligned}
\end{equation}
For $d=1$, 
$d(n-2k-2+d)-\floor{\frac{d(k+1-d)}{2}}=n-2k-1-\floor{k/2}\ge 0$.
For $d\ge 2$, 
$d(n-2k-2+d)-\floor{\frac{d(k+1-d)}{2}}
\ge \frac{d}{2}(2n-5k+1)\ge 0$.
Thus, 
$J_k(n)=(k+1)(n-k-1)$ for $n\ge \ceil{(5k+1)/2}$.
\renewcommand{\qedsymbol}{$\blacksquare$}
\end{proof}

Next, we give a piecewise formula for $F_k(n)$.
\begin{lemma}\label{lem:B-equals-f}
Let $k\ge 2$ and $n\ge k+2$. Then
$$F_k(n)=
\begin{cases}
\displaystyle \left\lfloor \frac{(k+1)n}{2}\right\rfloor,
& \displaystyle k+2\le n\le 
\left\lfloor (1+\frac{\sqrt{3}}{2})(k+1)\right\rfloor,\\
\displaystyle \left\lfloor \frac{(2n+k+1)^2}{24}\right\rfloor,
& \displaystyle 
\left\lfloor (1+\frac{\sqrt{3}}{2})(k+1)\right\rfloor
< n<
\ceil{\frac{5k+1}{2}},\\
\displaystyle (k+1)(n-k-1),
& \displaystyle n\ge
\ceil{\frac{5k+1}{2}}.
\end{cases}
$$
\end{lemma}
\begin{proof}

First suppose that
$k+2\le n\le \floor{(1+\frac{\sqrt{3}}{2})(k+1)}$.
By Lemma~\ref{lem:linear-bound}, the smallest integer $N_k\ge k+2$ for which $J_k(N_k)>R_k(N_k)$ satisfies $N_k\ge \lfloor(1+\frac{\sqrt3}{2})(k+1)\rfloor+1$.
Hence
$J_k(n)\le R_k(n)$.
Therefore
$
F_k(n)=R_k(n).
$

Next suppose that
$\floor{(1+\frac{\sqrt{3}}{2})(k+1)}<n<\ceil{(5k+1)/2}$.
For $k\in\{2,3\}$, this interval contains no integer.
Then we may assume
that $k\ge4$.
Let $a_0\in I_k$ be such that
$J_k(n)=a_0(n-a_0)+\left\lfloor\frac{a_0(k+1-a_0)}{2}\right\rfloor$.
Then
$$J_k(n)\le a_0(n-a_0)+\frac{a_0(k+1-a_0)}{2}
      =\frac{(2n+k+1)^2}{24}-\frac32\Bigl(a_0-\frac{2n+k+1}{6}\Bigr)^2
      \le\frac{(2n+k+1)^2}{24}.$$
Consequently
$J_k(n)\le \Bigl\lfloor\frac{(2n+k+1)^2}{24}\Bigr\rfloor. $

Since $\bigl\lfloor\bigl(1+\frac{\sqrt3}{2}\bigr)(k+1)\bigr\rfloor < n < \bigl\lceil\frac{(5k+1)}{2}\bigr\rceil$,
every integer closest to $\frac{(2n+k+1)}{6}$ is in $I_k$.
Choose an integer $a\in I_k$ nearest to $\frac{2n+k+1}{6}$, and set
$\epsilon:=|a-\frac{2n+k+1}{6}|$.
Then
$\epsilon\in\left\{0,\frac16,\frac13,\frac12\right\}$ and
$$
J_k(n)\ge a(n-a)+\Bigl\lfloor\frac{a(k+1-a)}{2}\Bigr\rfloor
       =\Bigl\lfloor\frac{(2n+k+1)^2}{24}-\frac32{\epsilon}^2\Bigr\rfloor. 
$$

In the following, we 
divide the proof into four cases according to the value of $\epsilon$.
If $\epsilon=0$, then
$\frac{3}{2}\epsilon^2=0$.
If $\epsilon=\frac{1}{6}$, then
$\frac{3}{2}\epsilon^2=\frac{1}{24}$ and the fractional part of
$\frac{(2n+k+1)^2}{24}$ is $\frac{1}{24}$.
If $\epsilon=\frac{1}{3}$, then
$\frac{3}{2}\epsilon^2=\frac{1}{6}$ and the fractional part of
$\frac{(2n+k+1)^2}{24}$ is $\frac{1}{6}$ or $\frac{2}{3}$.
If $\epsilon=\frac{1}{2}$, then
$\frac{3}{2}\epsilon^2=\frac{3}{8}$ and the fractional part of
$\frac{(2n+k+1)^2}{24}$ is $\frac{3}{8}$.

Therefore,
$\Bigl\lfloor\frac{(2n+k+1)^2}{24}-\frac32 {\epsilon}^2\Bigr\rfloor
   =\Bigl\lfloor\frac{(2n+k+1)^2}{24}\Bigr\rfloor$.
Thus
$J_k(n)=\left\lfloor\frac{(2n+k+1)^2}{24}\right\rfloor$.
Moreover, $n>\bigl(1+\frac{\sqrt3}{2}\bigr)(k+1)$, and hence
$\frac{(2n+k+1)^2}{24}>\frac{(k+1)n}{2}$.
It follows that $J_k(n)\ge R_k(n)$.
Hence $F_k(n)=J_k(n)$.

Finally, suppose that
$n\ge \ceil{(5k+1)/2}$.
By Lemma~\ref{lem:join-term-linear-range}, we have $J_k(n)=(k+1)(n-k-1)$. Since $n\ge \left\lceil(5k+1)/2\right\rceil$, we have
$n\ge2k+2$ for $k\ge2$. Consequently $J_k(n)\ge R_k(n)$, and hence
$F_k(n)=J_k(n)=(k+1)(n-k-1)$.
\renewcommand{\qedsymbol}{$\blacksquare$}
\end{proof}

To bound the number of edges of $G$, we decompose $G$ into edge-disjoint subgraphs.
The following inequalities combine the bounds for edge-disjoint subgraphs of $G$.

\begin{lemma}\label{lem:merge-ineq}
Let $k\ge2$. If $k+2\le n\le \lceil(5k+1)/2\rceil$ and
$J_k(n)>R_k(n)$, then the following inequalities hold.
\begin{enumerate}[label=\textup{(\roman*)},nosep]
    \item For all positive integers $x$ and $y$ satisfying $x+y-1\le n$, $F_k(x)+F_k(y)\le F_k(x+y-1)$;
    \item For all positive integers $x$ and $y$ satisfying  $x+y\le n$, $F_k(x)+F_k(y)\le F_k(x+y)$.
\end{enumerate}
\end{lemma}

\begin{proof}
When $k=2$, the hypothesis $J_k(n)>R_k(n)$ does not hold for
$4\le n\le6$. Thus we may assume that $k\ge3$.

(i) For $m\ge k+2$, by definition of $F_k$, $F_k(m)=\max\{R_k(m),J_k(m)\}\ge J_k(m)$.
Hence for $k+2\le m\le n$, we have either $F_k(m)=J_k(m)$ or $F_k(m)>J_k(m)$.

\begin{claim}\label{claim:F_k(n)-merge}
For every integer $m$ with $k+2\le m\le n$ and every $a\in I_k$,
if $F_k(m)>J_k(m)$, then $F_k(m)\le a(m-1)$.
\end{claim}
\begin{proof}
Since $F_k(m)>J_k(m)$, we have $F_k(m)=R_k(m)$.
Since $m\ge k+2$, we have 
$\frac{(k+1)m}{2}
\le
\frac{k+2}{2}(m-1)
\le
\left(\left\lfloor\frac{k+1}{2}\right\rfloor+1\right)(m-1)
\le a(m-1)$.
Hence
$F_k(m)=R_k(m)=\left\lfloor\frac{(k+1)m}{2}\right\rfloor\le a(m-1)$.
\end{proof}

\noindent\textbf{Case 1.} $1\le x,y\le k+1$.

If $x+y-1\le k+1$, then
$F_k(x)+F_k(y)=
\binom{x}{2}+\binom{y}{2}\le \binom{x+y-1}{2}=F_k(x+y-1).
$
If $x+y-1\ge k+2$, by symmetry we may assume $x\ge y$. 
If $\min\{x,y\}=1$, then the assertion is immediate. We may therefore assume that $2\le x,y\le k+1$.
Since $x+y$ is fixed, 
$\binom{x}{2}+\binom{y}{2}$
is maximized at $x=k+1$ and $y=x+y-k-1$. Hence
$\binom{x}{2}+\binom{y}{2}
\le
\binom{k+1}{2}+\binom{x+y-k-1}{2}$.
Moreover,
$\frac{(k+1)(x+y-1)}{2}
-\binom{k+1}{2}
-\binom{x+y-k-1}{2} =
\frac{(x+y-k-1)(2k+3-x-y)}{2}\ge 0$.
Therefore
$\binom{k+1}{2}+\binom{x+y-k-1}{2}
\le R_k(x+y-1)$.
Consequently,
$F_k(x)+F_k(y)=\binom{x}{2}+\binom{y}{2}\le R_k(x+y-1)\le F_k(x+y-1)$.

\medskip
\noindent\textbf{Case 2.} $1\le x\le k+1$ and $y\ge k+2$.

Then $x+y-1\ge k+2$.
If $F_k(y)=J_k(y)$, let $b\in I_k$ be the integer such that $J_k(y)=b(y-b)+\left\lfloor\frac{b(k+1-b)}{2}\right\rfloor$.
Thus,
$$\begin{aligned}
F_k(x)+F_k(y)&=\binom{x}{2}+J_k(y)
=\binom{x}{2}+b(y-b)+\left\lfloor\frac{b(k+1-b)}{2}\right\rfloor\\
&\le b(x-1)+b(y-b)+\left\lfloor\frac{b(k+1-b)}{2}\right\rfloor
=b(x+y-1-b)+\left\lfloor\frac{b(k+1-b)}{2}\right\rfloor\\
&\le J_k(x+y-1)\le F_k(x+y-1).
\end{aligned}
$$

If $F_k(y)=R_k(y)$, then
$$\begin{aligned}
    F_k(x)+F_k(y)&\le
\binom{x}{2}+\left\lfloor\frac{(k+1)y}{2}\right\rfloor
\le \left\lfloor\frac{(k+1)(x+y-1)}{2}\right\rfloor\\
&=R_k(x+y-1)\le F_k(x+y-1).
\end{aligned}
$$

\medskip
\noindent\textbf{Case 3.} $x,y\ge k+2$.

First assume that exactly one of $F_k(x)=J_k(x)$ and $F_k(y)=J_k(y)$ holds.
By symmetry, assume
that
$F_k(x)>J_k(x)$ and $F_k(y)=J_k(y)$.
Then there is an integer $b\in I_k$ such that
$F_k(y)=J_k(y)=b(y-b)+\left\lfloor\frac{b(k+1-b)}{2}\right\rfloor$.
By Claim~\ref{claim:F_k(n)-merge} we have
$F_k(x)\le b(x-1)$.
Hence
$$
F_k(x)+F_k(y)\le b(x-1)+b(y-b)+\left\lfloor\frac{b(k+1-b)}{2}\right\rfloor
\le J_k(x+y-1)\le F_k(x+y-1).
$$

Next suppose that $F_k(x)=J_k(x)$ and $F_k(y)=J_k(y)$. 
Let $a,b\in I_k$ be such that $F_k(x)=J_k(x)=a(x-a)+\left\lfloor\frac{a(k+1-a)}{2}\right\rfloor$
and
$F_k(y)=J_k(y)=b(y-b)+\left\lfloor\frac{b(k+1-b)}{2}\right\rfloor$.
We may assume that $a\le b$.
Since $J_k(x+y-1)\ge b(x+y-1-b)+\left\lfloor\frac{b(k+1-b)}{2}\right\rfloor$,
we have $J_k(x+y-1)-J_k(y)\ge b(x-1)$.
Because $a\le b$ and $x\ge 1$, we obtain
$$
\begin{aligned}
F_k(x)+F_k(y)-J_k(x+y-1)
&\le a(x-a)+\frac{a(k+1-a)}{2}-b(x-1)\\
&=(a-b)(x-1)+\frac{a(k+3-3a)}{2}\le0.
\end{aligned}
$$
Hence
$F_k(x)+F_k(y)\le J_k(x+y-1)\le F_k(x+y-1)$.

Finally, assume that $F_k(x)>J_k(x)$ and $F_k(y)>J_k(y)$.
Since $x+y\ge 2k+4$, we have
$$
\begin{aligned}
F_k(x)+F_k(y)
&=R_k(x)+R_k(y)\le R_k(x+y)\le \frac{(k+1)(x+y)}{2}\\
&\le (k+1)(x+y-k-2)
\le J_k(x+y-1)\le F_k(x+y-1).
\end{aligned}
$$

(ii) We first show that $F_k(n)$ is strictly increasing.
For $1\le n\le k$, $F_k(n+1)-F_k(n)=n\ge1$. For $n=k+1$, $F_k(k+2)\ge R_k(k+2)>\binom{k+1}{2}=F_k(k+1)$. For $n\ge k+2$, $R_k(n+1)\ge R_k(n)+1$ and $J_k(n+1)\ge J_k(n)+1$, so $F_k(n+1)=\max\{R_k(n+1),J_k(n+1)\}>\max\{R_k(n),J_k(n)\}=F_k(n)$. Hence $F_k(n+1)>F_k(n)$ for all $n\ge1$.

Since $x+y\le n$, part~(i) gives
$F_k(x)+F_k(y)\le F_k(x+y-1)\le F_k(x+y)$.
\renewcommand{\qedsymbol}{$\blacksquare$}
\end{proof}

The next two lemmas provide the edge bounds used in the proof of Lemma~\ref{lem:base}.

\begin{lemma}\label{lem:cycle-arithmetic-consequences}
Let $k\ge 2$ and
$k+2\le n\le \ceil{{(5k+1)}/{2}}$.
Let $c<n$ be an integer. Then the following statements hold:
\begin{enumerate}[label=\textup{(\roman*)},nosep]
\item\label{lem:cycle-arithmetic-i}
For every integer $s$ with $2\le s\le \floor{\frac{c}{2}}-1$, we have
$s(n-c)+\floor{\frac{(s-1)s+(c-s+1)(k+1)}{2}}\le F_k(n)$.

\item\label{lem:cycle-arithmetic-ii}
If $\floor{c/2}\ge1$, then
$\left(\floor{\frac{c}{2}}-1\right)(n-c)+\floor{\frac{\left(\floor{\frac{c}{2}}-1\right)\min\left\{\floor{\frac{c}{2}},k+1\right\}+\left(c-\floor{\frac{c}{2}}+1\right)(k+1)}{2}}\le F_k(n)$.
\end{enumerate}
\end{lemma}

\begin{proof}
We first consider $k=2$. In this case $4\le n\le6$. Since $c<n$, we have
$\floor{c/2}-1\le1$.
Thus there is no integer $s$ in the stated range in
\textup{(i)}.

For \textup{(ii)}, since $k=2$, we have $4\le n\le 6$. We divide the proof into three cases.

Suppose that $n=4$.
We have $c\in\{2,3\}$. The left-hand side takes the values $3$ and $4$, respectively. Since $F_2(4)=6$, the inequality in (ii) holds.

Suppose that $n=5$.
We have $c\in\{2,3,4\}$. The left-hand side takes the values $3$, $4$, and $6$, respectively. Since $F_2(5)=7$, the inequality in (ii) holds.

Suppose that $n=6$.
We have $c\in\{2,3,4,5\}$. The left-hand side takes the values $3$, $4$, $7$, and $8$, respectively. Since $F_2(6)=9$, the inequality in (ii) holds.

Therefore, both \textup{(i)} and \textup{(ii)} hold when $k=2$.

Assume that $k\ge3$. 
Let $\varepsilon\in\{0,1\}$ and let $r$ be an integer satisfying
$1+\varepsilon\le r\le\lfloor c/2\rfloor$. Set
$$L_{\varepsilon}(c,r)=(r-1)(n-c)+\floor{\frac{(r-1-\varepsilon)(r-\varepsilon)+(c-r+1+\varepsilon)(k+1)}{2}}.$$
We now claim that
$L_{\varepsilon}(c,r)\le F_k(n)$.

We first compare $L_{\varepsilon}(c,r)$ with $\frac{(k+1)n}{2}$ using the following inequality:
\begin{equation}\label{eq:L-versus-regular}
2\left(L_{\varepsilon}(c,r)-\frac{(k+1)n}{2}\right)\le (n-c)(2r-k-3)-(r-1-\varepsilon)(k+1-r+\varepsilon).
\end{equation}
If the right-hand side of \eqref{eq:L-versus-regular} is nonpositive,
then since $L_{\varepsilon}(c,r)$ is an integer, we have
$L_{\varepsilon}(c,r)\le R_k(n)$.
We distinguish two cases.

\medskip
\noindent\textbf{Case 1.}
$J_k(n)\le R_k(n)$.

Then $F_k(n)=R_k(n)$. By Lemma~\ref{lem:linear-bound}, we have $n\le2k+2$. Since $2r\le c<n$, it follows that $r\le k$.

First, suppose that $2r-k-3\le0$. 
Since $n-c>0$, $r-1-\varepsilon\ge0$, and $k+1-r+\varepsilon\ge0$, both terms on the right-hand side of \eqref{eq:L-versus-regular} are nonpositive. 
Consequently, $L_{\varepsilon}(c,r)\le R_k(n)=F_k(n)$.

Next, we assume that $2r-k-3>0$. Then $r\ge\floor{\frac{k+1}{2}}+1$, and hence $r\in I_k$. Therefore
$$
\begin{aligned}
r(n-r)+\floor{\frac{r(k+1-r)}{2}} \le J_k(n)\le R_k(n)\le\frac{(k+1)n}{2}.
\end{aligned}
$$
Since $\floor{\frac{r(k+1-r)}{2}}\ge\frac{r(k+1-r)}{2}-\frac12$, 
we obtain $r(n-r)+{\frac{r(k+1-r)}{2}}-\frac{1}{2} \le\frac{(k+1)n}{2}$.
Thus, $(2r-k-1)n\le2r^2-r(k+1-r)+1$.
Hence
\begin{equation}\label{eq:n-upper-bound}
n\le\frac{2r^2-r(k+1-r)+1}{2r-k-1}.
\end{equation}

We claim that for $\varepsilon\in\{0,1\}$, $(r-1-\varepsilon)(k+1-r+\varepsilon)>(2r-k-3)(n-2r)$. 
Indeed, from \eqref{eq:n-upper-bound} and $2r-k-3>0$, 
we have
$(2r-k-3)(n-2r)\le \frac{(2r-k-3)\big(r(k+1-r)+1\big)}{2r-k-1}.$
Thus, for $\varepsilon=0$,
\begin{equation}\label{eq:varepsilon-zero-positive}
(r-1)(k+1-r)-(2r-k-3)(n-2r)
\ge \frac{r(k-r)+(k+1-r)^2+(k+1-r)+2}{2r-k-1},
\end{equation}
and for $\varepsilon=1$,
\begin{equation}\label{eq:varepsilon-one-positive}
(r-2)(k+2-r)-(2r-k-3)(n-2r)
\ge \frac{r(2r-k-4)+2(k+1-r)^2+3(k+1-r)+2}{2r-k-1}.
\end{equation}

Since $2r-k-3>0$, the common denominator $2r-k-1$ in
\eqref{eq:varepsilon-zero-positive} and \eqref{eq:varepsilon-one-positive} is positive.
Since $r\le k$,
the numerator in \eqref{eq:varepsilon-zero-positive} satisfies
$$
r(k-r)+(k+1-r)^2+(k+1-r)+2=(k+3)(k-r)+4\ge4.
$$
Since $k\ge 3$,
the numerator in \eqref{eq:varepsilon-one-positive} satisfies
$$
\begin{aligned}
&r(2r-k-4)+2(k+1-r)^2+3(k+1-r)+2\\
&\qquad=4r^2-(5k+11)r+2k^2+7k+7\\
&\qquad=4\left(r-\frac{5k+11}{8}\right)^2
+\frac{(k-1)(7k+9)}{16}>0.
\end{aligned}
$$
Hence, the right-hand sides of \eqref{eq:varepsilon-zero-positive} and \eqref{eq:varepsilon-one-positive} are positive.
Then, for $\varepsilon\in\{0,1\}$, $(r-1-\varepsilon)(k+1-r+\varepsilon)>(2r-k-3)(n-2r)$. 

Since $c\ge2r$, we have $n-c\le n-2r$. It follows that $(2r-k-3)(n-c)\le(r-1-\varepsilon)(k+1-r+\varepsilon)$. The right-hand side of \eqref{eq:L-versus-regular} is  nonpositive.
Hence $L_{\varepsilon}(c,r)\le R_k(n)=F_k(n)$.

\medskip
\noindent\textbf{Case 2.}
$J_k(n)>R_k(n)$.

If $k=3$, then $5\le n\le8$ and $R_3(n)=2n$, $J_3(n)=\max\{3n-8,4n-16\}\le2n$. Thus, this case cannot occur when $k=3$.
Therefore, we may assume that $k\ge4$.
Then $F_k(n)=J_k(n)$, and
Lemma~\ref{lem:linear-bound} gives
$n>\left(1+\frac{\sqrt3}{2}\right)(k+1)$.
If $2r-k-3\le0$, then $r\le k$, and the right-hand side of
\eqref{eq:L-versus-regular} is nonpositive. Hence
$L_{\varepsilon}(c,r)\le R_k(n)<J_k(n)=F_k(n)$.
We may therefore assume that $2r-k-3>0$. In the following we prove that
$L_{\varepsilon}(c,r)\le J_k(n)$.

We first reduce the problem to the case $c=2r$. 
Assume $c>2r$.
Since decreasing $c$ by one decreases the numerator inside the floor of $L_\varepsilon(c,r)$ by
$k+1$, we have
$L_\varepsilon(c-1,r)-L_\varepsilon(c,r)
\ge
(r-1)-\left\lceil\frac{k+1}{2}\right\rceil$.
Since $2r-k-3>0$, we have
$r-1\ge\ceil{\frac{k+1}{2}}$.
Then $L_{\varepsilon}(c-1,r)\ge L_{\varepsilon}(c,r)$. 

We can repeat this argument as $c$ decreases to $2r$.  
Then $L_\varepsilon(c,r)\le L_\varepsilon(2r,r)$.
By dropping the floor, we have
\begin{equation}\label{eq:L}
    L_\varepsilon(c,r)\le L_\varepsilon(2r,r)
\le
(r-1)(n-2r)
+
\frac{
(r-1-\varepsilon)(r-\varepsilon)
+(r+1+\varepsilon)(k+1)
}{2}.
\end{equation}

We next find a lower bound for $J_k(n)$. 
Note that
$$a(n-a)+\frac{a(k+1-a)}{2}
=
\frac{(2n+k+1)^2}{24}
-\frac{3}{2}\left(a-\frac{2n+k+1}{6}\right)^2.$$
Since $\bigl\lfloor\bigl(1+\frac{\sqrt3}{2}\bigr)(k+1)\bigr\rfloor \le n \le \bigl\lceil\frac{(5k+1)}{2}\bigr\rceil$,
every integer closest to $\frac{(2n+k+1)}{6}$ is in $I_k$.
Choose an integer $a\in I_k$ nearest to $\frac{2n+k+1}{6}$.
Then $\lvert a-\frac{2n+k+1}{6} \rvert\le \frac12$.
It implies that $a(n-a)+\frac{a(k+1-a)}{2}\ge \frac{(2n+k+1)^2}{24}
-\frac{3}{8}$.
Since $a(n-a)+\frac{a(k+1-a)}{2}$ is an integer or a half-integer,
we have $J_k(n)\ge\floor{a(n-a)+\frac{a(k+1-a)}{2}}\ge a(n-a)+\frac{a(k+1-a)}{2}-\frac12$.
Therefore,
\begin{equation}\label{eq:join-term-lower-bound}
J_k(n)
\ge
\frac{(2n+k+1)^2}{24}
-\frac78.
\end{equation}

First, suppose that $\varepsilon=0$. By \eqref{eq:L} and \eqref{eq:join-term-lower-bound}, $J_k(n)-L_0(c,r)\ge\frac n2-\frac{3(k+1)}4-\frac54$. If $k\ge6$, then
$J_k(n)-L_0(c,r) > \frac{\sqrt3-1}{4}(k+1)-\frac54\ge\frac{7\sqrt3-12}{4}>0$.
Thus, $L_0(c,r)<J_k(n)$ for $k\ge6$.
Assume that $k\in\{4,5\}$. Since $2r-k-3>0$ and $2r\le c<n\le\ceil{\frac{5k+1}{2}}$, we have $r\in\{k,k+1\}$.

If $r=k$, 
then $n\ge10$ when $k=4$ and $n\ge12$ when $k=5$.
We have $L_0(2k,k)=(k-1)n-k^2+2k+\floor{\frac{k+1}{2}}$. On the other hand, by the definition of $J_k(n)$, taking $a=k$ gives $J_k(n)\ge kn-k^2+\floor{\frac{k}{2}}$. Hence $J_k(n)-L_0(2k,k)\ge n-2k+\floor{\frac{k}{2}}-\floor{\frac{k+1}{2}}\ge1$. 

If $r=k+1$, then $L_0(2k+2,k+1)=kn-k^2+1\le kn-k^2+\floor{\frac{k}{2}}\le J_k(n)$. Therefore,  $L_0(c,r)\le J_k(n)$.

Now suppose that $\varepsilon=1$. By \eqref{eq:L} and \eqref{eq:join-term-lower-bound}, $J_k(n)-L_1(c,r)\ge\frac{5n}{6}-\frac{13(k+1)}{12}-\frac{23}{12}$. If $k\ge4$,
$J_k(n)-L_1(c,r,1) > \frac{5\sqrt3-3}{12}(k+1)-\frac{23}{12}\ge\frac{25\sqrt3-38}{12}>0.$
Thus $L_1(c,r)<J_k(n)$.

\medskip
Combining Cases~1 and~2, we obtain $L_{\varepsilon}(c,r)\le F_k(n)$ for every $k\ge3$ and all possible $r$ and $\varepsilon$.

For (i), take $\varepsilon=1$ and $r=s+1$. Since $2\le s\le\floor{c/2}-1$,  $s(n-c)+\floor{\frac{(s-1)s+(c-s+1)(k+1)}{2}}\le F_k(n)$. This proves \textup{(i)}.

For (ii), take $\varepsilon=0$ and $r=\lfloor c/2\rfloor$. Since $r\ge1$,  $(r-1)(n-c)+\floor{\frac{(r-1)r+(c-r+1)(k+1)}{2}}\le F_k(n)$. If $r\le k+1$, then $\min\{r,k+1\}=r$. 
Then the inequality in \textup{(ii)} follows immediately. If $r>k+1$, then $(r-1)(k+1)\le(r-1)r$.
Hence $\floor{\frac{(r-1)(k+1)+(c-r+1)(k+1)}{2}}\le\floor{\frac{(r-1)r+(c-r+1)(k+1)}{2}}$. This proves \textup{(ii)}.
\renewcommand{\qedsymbol}{$\blacksquare$}
\end{proof}

\begin{lemma}\label{lem:small-c-arithmetic}
Let $2\le k\le 6$, let
$k+2\le n\le\lceil{(5k+1)}/{2}\rceil$,
and let $c$ be an integer satisfying
$2\lfloor{(k+1)}/{2}\rfloor+2\le c<n$
and 
$c\le 9$.
Then
\begin{equation}\label{eq:small-c-arithmetic}
\Bigl\lfloor \frac{c}{2}\Bigr\rfloor(n-c)+
\Bigl\lfloor\frac{(\lfloor \frac{c}{2}\rfloor-1)\lfloor \frac{c}{2}\rfloor+(c-\lfloor \frac{c}{2}\rfloor+1)(k+1)}{2}\Bigr\rfloor
\le F_k(n),
\end{equation}
except when $(k,n,c)=(3,8,6)$. 
If $(k,n,c)=(3,8,6)$, the left-hand
side of \eqref{eq:small-c-arithmetic} is $17=F_3(8)+1$.
\end{lemma}

\begin{proof}
We first list all possible values of $(k,c)$ in Table~\ref{tab:smallk}:
\begin{table}[H]
\centering
\begin{tabular}{c|c|c}
$k$ & $c$ & \text{left-hand side of \eqref{eq:small-c-arithmetic}}\\ \hline
2 & 4,5 & $2n-3$\\
3 & 6,7 & $3n-(c+1)$\\
4 & 6,7   & $3n-(c-1)$\\
4 & 8,9   & $4n-(c+6)$\\
5 & 8,9   & $4n-(c+3)$\\
6 & 8,9 & $4n-9$
\end{tabular}
\caption{$2\le k\le 6$}
\label{tab:smallk}
\end{table}

Suppose that $k=2$.
For $n=5$, $F_2(5)=\max\{\lfloor 15/2\rfloor,\,2\cdot5-3\}=7$.
For $n=6$, $F_2(6)=\max\{9,\,9\}=9$.  Hence $2n-3=F_2(n)$.

Suppose that $k=3$.
The possible pairs $(n,c)$ are $(7,6)$, $(8,6)$, $(8,7)$. The
left-hand side gives $14$, $17$, $16$, while $F_3(7)=14$ and
$F_3(8)=16$.  The only exception is $(k,n,c)=(3,8,6)$.

Suppose that $k=4$.
If $c=6$, the left-hand side of \eqref{eq:small-c-arithmetic} takes value $3n-5$.  For $n\le 10$, we have
$3n-5\le\lfloor5n/2\rfloor=R_4(n)\le F_4(n)$.
For $n=11$, we have
$3\cdot11-5=28<30=F_4(11)$.
If $c=7$, then $3n-6\le\lfloor5n/2\rfloor\le F_4(n)$.
If $c\in\{8,9\}$, the left-hand side of \eqref{eq:small-c-arithmetic} takes values $4n-14$ or $4n-15$. 
Hence \eqref{eq:small-c-arithmetic} holds because 
$J_{4}(n)=4n-14\le F_4(n)$.

Suppose that $k=5$.
For $n\le 11$, we have $4n-11\le 3n\le F_5(n)$.
For $n\ge 12$, both $4n-11$ and $4n-12$ are at most
$J_{5}(n)=5n-23\le F_5(n)$.

Suppose that $k=6$.
Since $n\le 16$, we obtain
$4n-9\le\bigl\lfloor\frac{7n}{2}\bigr\rfloor=R_6(n)\le F_6(n)$.

This completes the proof.
\renewcommand{\qedsymbol}{$\blacksquare$}
\end{proof}

\section{Tools: closure, edge-switching, and stability}\label{sec:tools}

We collect the closure, edge-switching, and circumference results used in the proof of Theorem~\ref{thm:main}. These results will be used to bound the number of edges inside $C$ and the number of edges with at least one endpoint outside $C$.

\subsection{Three families of graphs}\label{sec:graphs}

We use the graph $W_{n,s,t}$ and the exceptional families
$\mathcal X_{n,c}$ and $\mathcal Y_{n,c}$ from the stability results
in \cite{FKLV2018,FKV2016,MaNing2020}.
The families $\mathcal X_{n,c}$ and $\mathcal Y_{n,c}$ are defined
for odd $c$.

\begin{itemize}
    \item For integers $n,s,t$ with $1\le s\le \floor{{(t+1)}/{2}}$ and $t-s+1\le n$, define $W_{n,s,t}$ as follows. Let $S,T,U$ be pairwise disjoint sets with $\lvert S\rvert=s$, $\lvert T\rvert=t-2s+1$, and $\lvert U\rvert=n-t+s-1$. Make $S\cup T$ a clique, make $U$ independent, and join each vertex of $U$ to the vertices of $S$. The case $U=\emptyset$ is allowed. Equivalently, $W_{n,s,t}$ is obtained from a clique $K_{t-s+1}$ by adding $n-(t-s+1)$ vertices whose common neighborhood is a fixed $s$-set in the clique. Thus $e(W_{n,s,t})=\binom{t-s+1}{2}+s(n-t+s-1)$.
    \item A graph $G$ belongs to $\mathcal X_{n,c}$ if it has order $n$
    and a partition $A\cup B\cup X$ such that
$G[A]$ induces a clique $K_{\lfloor{c/2}\rfloor}$,
   both $B$ and $X$ are
    independent, $(A,B)$ is complete bipartite, and there exist
    $a\in A$ and $b\in B$ such that $N_G(x)=\{a,b\}$ for every
    $x\in X$.
    \item A graph $G$ belongs to $\mathcal Y_{n,c}$ if it has order $n$
    and a partition $A\cup B\cup Y$ such that
$G[A]$ induces a clique $K_{\lfloor{c/2}\rfloor}$,
   $B$ is
    independent, $(A,B)$ is complete bipartite, and $G[Y]$ is a nontrivial
    star forest (with at least two components, each containing an edge). There are vertices
    $a,b\in A$ such that every component $R$ of $G[Y]$ satisfies
    $N_G(V(R))=\{a,b\}$. If $\lvert V(R)\rvert\ge3$, then all leaves of $R$ have
    degree $2$ in $G$ and have a common neighbor in $\{a,b\}$.
\end{itemize}

\subsection{Classical and stability tools}
\begin{theorem}[Dirac~\cite{Dirac1952}]\label{thm:Dirac-cycle}
If $G$ is a $2$-connected graph with minimum degree $\delta(G)$, then $c(G)\ge \min\{\lvert V(G)\rvert,2\delta(G)\}$. In particular, if $G$ is non-Hamiltonian, then $c(G)\ge 2\delta(G)$.
\end{theorem}

\begin{theorem}[Ma--Ning~\cite{MaNing2020}]\label{thm:MN-exterior-stability}
Let $G$ be a $2$-connected graph on $n$ vertices, and let $C$ be a longest
cycle of length $c$, where $10\le c\le n-1$. If
$$
e(G-C)+e_G(V(G)\setminus V(C),V(C))
  >(\floor{c/2}-1)(n-c),
$$
then either $G\subseteq W_{n,\floor{c/2},c}$, or $c$ is odd and $G$ is a
subgraph of a member of $\mathcal{X}_{n,c}\cup\mathcal{Y}_{n,c}$.
\end{theorem}

\begin{lemma}\label{lem:exceptional-low-degree-or-W}
Let $c<n$ be odd and let
$G\in \mathcal{X}_{n,c}\cup\mathcal{Y}_{n,c}$. 
Then either $G\subseteq W_{n,\floor{c/2},c}$, or $G$ has a vertex of degree at most $3$.
\end{lemma}

\begin{proof}
First, suppose that $G\in\mathcal{X}_{n,c}$, with partition
$V(G)=A\cup B\cup X_0$. If $X_0\ne\emptyset$, then every vertex of $X_0$ has neighborhood exactly $\{a,b\}$ and hence degree $2$.

Thus assume $X_0=\emptyset$. Then $V(G)=A\cup B$ such that $G[A]$ induces a clique $K_{\lfloor{c/2}\rfloor}$,
   $B$ is
    independent, $(A,B)$ is complete bipartite. Since $c$ is odd and $n>c$, we have
$\lvert B\rvert=n-\floor{c/2}\ge \floor{c/2}+2\ge2$. Choose distinct
$b_1,b_2\in B$. In the partition of $W_{n,\floor{c/2},c}$,
take
$S=A$, $T=\{b_1,b_2\}$, $U=B\setminus\{b_1,b_2\}$.
Indeed, $\lvert T\rvert=c-2\floor{c/2}+1=2$. Then every edge of $G$ is an edge of
this copy of $W_{n,\floor{c/2},c}$. Hence
$G\subseteq W_{n,\floor{c/2},c}$.

Now suppose that $G\in\mathcal{Y}_{n,c}$, with partition
$V(G)=A\cup B\cup Y_0$. If some star component of $G[Y_0]$ has order at
least $3$, then one of its leaves has degree $2$ in $G$. Otherwise every
star component of $G[Y_0]$ has order $2$. For any such component $uv$,
the condition $N_G(\{u,v\})=\{a,b\}$ gives
$d_G(u),d_G(v)\le3$. Thus $G$ has a vertex of degree at most $3$.
\renewcommand{\qedsymbol}{$\blacksquare$}
\end{proof}

\subsection{Some results on cycles}

The following lemma is a slight variant of a lemma of Ma and Ning~\cite[Lemma~2.11]{MaNing2020}.
It will be used to obtain the edge bound.

\begin{lemma}[{\normalfont Variant of Ma--Ning~\cite[Lemma~2.11]{MaNing2020}}]
\label{lem:Ma-Ning-cycle-dichotomy}
Let $G$ be a $2$-connected graph and let $C$ be a locally maximal cycle of
length $c<\lvert V(G)\rvert$. Set $\widehat G=\cl_C(G)$ and $\widehat H=\widehat G[V(C)]$.
Then one of the following holds.
\begin{enumerate}[label=\textup{(\roman*)},nosep]
\item There exist an integer $s$ and a set $S\subseteq V(C)$ such that
$2\le s\le \floor{c/2}-1$, $\lvert S\rvert=s-1$, $d_{\widehat H}(u)\le s$ for every $u\in S$,
and ${\widehat H}-S$ is a clique.
\item There exists a set $R\subseteq V(C)$ such that
$\lvert R\rvert=\floor{c/2}-1$ and $d_{\widehat H}(u)\le \floor{c/2}$ for every $u\in R$.
\end{enumerate}
\end{lemma}

To prove the lemma, we need some results of Ma and Ning~\cite{MaNing2020}. 

\begin{lemma}[{\normalfont Ma--Ning~\cite[Lemma~2.7]{MaNing2020}}]
\label{lem:closure-preserves-local-maximality}
Let $G$ be a graph and let $C$ be a locally maximal cycle of $G$. Then $C$ is locally maximal in $\cl_C(G)$.
\end{lemma}

\begin{lemma}[{\normalfont Ma--Ning~\cite[Lemma~2.8]{MaNing2020}}]
\label{lem:Ma-Ning-non-ha}
Let $G$ be a $2$-connected graph on $n$ vertices and let $C$ be a locally maximal cycle of
length $c<n$. Then $\cl_C(G)[V(C)]$ is non-Hamiltonian-connected.
\end{lemma}

\begin{lemma}[{\normalfont Ma--Ning~\cite[Lemma~2.10]{MaNing2020}}]
\label{lem:Ma-Ning-degree}
Let $G$ be a non-Hamiltonian-connected graph on $n$ vertices with minimum
degree at least $2$. Then there exists an integer $s$ with
$2\le s\le \floor{n/2}$ such that $G$ has at least $s-1$ vertices of degree
at most $s$.
\end{lemma}
We now prove Lemma~\ref{lem:Ma-Ning-cycle-dichotomy}.

\begin{proof}[\bf Proof of Lemma~\ref{lem:Ma-Ning-cycle-dichotomy}]
By Lemma~\ref{lem:Ma-Ning-non-ha}, ${\widehat H}$ is non-Hamiltonian-connected. Since
$C\subseteq {\widehat H}$, we have $\delta({\widehat H})\ge 2$. 
Thus
Lemma~\ref{lem:Ma-Ning-degree} applies.
Then there exists an integer $s$ with
$2\le s\le \floor{c/2}$ such that $\widehat H$ has at least $s-1$ vertices of degree
at most $s$. Take such an $s$ to be maximal.

If $s=\floor{c/2}$, then case \textup{(ii)} holds by taking any $\floor{c/2}-1$ of these vertices. Thus we may assume $s\le \floor{c/2}-1$.
Let $S_0=\{v\in V(C):d_{\widehat H}(v)\le s\}$. We show that either $\widehat H-S_0$ is a clique, which gives \textup{(i)}, or \textup{(ii)} holds. By the choice of $s$, we have $\lvert S_0\rvert\ge s-1$. If $\lvert S_0\rvert\ge s$, then there are at least $s$ vertices of degree at most $s$. Since $s+1\le \floor{c/2}$, this contradicts the maximality of $s$. Hence $\lvert S_0\rvert=s-1$.

If ${\widehat H}-S_0$ is a clique, then alternative \textup{(i)} holds. Thus, we may
assume that ${\widehat H}-S_0$ is not a clique. We show that
alternative \textup{(ii)} holds under this assumption.
Let $X=\{x\in V(\widehat H)\setminus S_0:\text{ $x$ has a non-neighbor in $\widehat H-S_0$}\}$.
Choose two nonadjacent vertices $u,v\in X$ such that $d_{\widehat H}(u)$ is maximum among all vertices in $X$. Set
$S'=V(C)\setminus (N_{\widehat H}(u)\cup\{u\})$ and $s'=\lvert S'\rvert+1=c-d_{\widehat H}(u)$.

For every $w\in S'$, $w$ and $u$ are nonadjacent. Since ${\widehat H}$ is $(c+1)$-closed, $d_{\widehat H}(w)+d_{\widehat H}(u)\le c$.
Hence $d_{\widehat H}(w)\le s'$. 
Since
$v\in S'\setminus S_0$, we have $s<d_{\widehat H}(v)\le s'$. If
$s'\le \floor{c/2}$, then ${\widehat H}$ has at least $s'-1$ vertices of degree
at most $s'$, contradicting the
maximality of $s$. Therefore $s'\ge \floor{c/2}+1$. 
It follows that $\lvert S'\rvert\ge \floor{c/2}$.
Since $d_{\widehat H}(u)=c-s'$, we have $d_{\widehat H}(u)\le c-\lfloor c/2\rfloor-1\le \lfloor c/2\rfloor$.
If $w\in S'\cap S_0$, then
$d_{\widehat H}(w)\le s\le \floor{c/2}$. If $w\in S'\setminus S_0$, then the choice
of $u$ implies $d_{\widehat H}(w)\le d_{\widehat H}(u)\le \floor{c/2}$. Hence, every vertex of
$S'$ has degree at most $\floor{c/2}$. Therefore, any $\floor{c/2}-1$
vertices of $S'$ form the desired set $R$ in alternative \textup{(ii)}.
\renewcommand{\qedsymbol}{$\blacksquare$}
\end{proof}

\subsection{The number of edges with at least one endpoint outside $C$}

The next lemma  
bounds the number of edges with at least one endpoint outside a
locally maximal cycle.
For a longest cycle $C$ of length $c$ in
an $n$-vertex graph, 
Bondy~\cite{Bondy1971} proved that (a) the number of edges with at most
one endpoint in $C$ is at most $\frac{c}{2}(n-c)$; (b) if the graph is
$2$-connected, the sharper bound
$\left\lfloor\frac{c}{2}\right\rfloor(n-c)$ holds.
In the following lemma,
we extend (b) from longest
cycles to locally maximal cycles.

\begin{lemma}\label{lem:exterior-edge-bound}
Let $G$ be a $2$-connected graph of order $n$, and let $C$ be a locally
maximal cycle of length $c<n$.
\begin{enumerate}[label=\textup{(\roman*)},nosep]
\item Let $\widehat G=\cl_C(G)$. If there exists a set $S\subseteq V(C)$
with $\lvert S\rvert=s-1$, $2\le s\le\bigl\lfloor\frac{c}{2}\bigr\rfloor-1$,
such that $\widehat G[V(C)]-S$ is a clique, then
$e(G)\le e(G[V(C)])+s(n-c)$.

\item If $c\ge4$, then
$e(G-C)+e_G\bigl(V(G)\setminus V(C),V(C)\bigr)
   \le\left\lfloor\frac{c}{2}\right\rfloor(n-c)$.
\end{enumerate}
\end{lemma}

To prove this, we will need some tools developed by Ma and Ning~\cite{MaNing2020} and Fan, Lv, and Wang~\cite{FanLvWang2004}.
Let $G$ be a graph, let $C$ be a cycle of $G$, and let $Q$ be a component of $G-C$.
A set $T=\{x_1,\ldots,x_t\}\subseteq V(C)$ with $t\ge2$, listed in cyclic order on $C$,
is a \textit{strong attachment} of $Q$ to $C$ if, for every $i\in\{1,\ldots,t\}$,
where indices are taken modulo $t$, there exist vertices $u_i,v_i\in V(Q)$ such that
$x_i u_i,x_{i+1}v_i\in E(G)$ and 
 $x_i u_i$, $x_{i+1}v_i$ are independent edges.

\begin{lemma}[{\normalfont Ma--Ning~\cite[Lemma~2.5(i)]{MaNing2020}}]\label{lem:Ma-Ning-strong-attachment}
Let $G$ be a $2$-connected graph, let $C$ be a locally maximal cycle of length $c$ in $G$, and let $Q$ be a component of $G-C$. Suppose $T$ is a strong attachment of $Q$ to $C$, $\lvert T\rvert=t$, 
and for every distinct $x,x'\in T$, 
a longest $(x,Q,x')$-path has length at least 
$d\ge2$. Then $\omega(G[V(C)])\le c-(d-1)(t-1)$.
\end{lemma}

Let $G$ be a graph and let $xy\in E(G)$. Suppose
$A\subseteq N_G(y)\setminus(N_G(x)\cup\{x\})$.
We define the edge switch from $y$ to $x$ over $A$ as follows. 
Let $G[y\to x;A]$ be obtained from $G$ by replacing $yz$ with $xz$ for every $z\in A$. Thus, the operation transfers the adjacencies of $y$ with the vertices in $A$ to $x$, while keeping the number of edges unchanged.

\begin{lemma}[{\normalfont Fan--Lv--Wang~%
\cite[Lemma~2.4]{FanLvWang2004}}]
\label{lem:Fan-Lv-Wang-switch}
Let $G$ be a $2$-connected graph, let $C$ be a locally maximal cycle
in $G$, and let $Q$ be a component of $G-C$. Then one of the following
holds.
\begin{enumerate}[label=\textup{(\roman*)},nosep]
\item
$N_Q(x)=V(Q)$ for every $x\in N_C(Q)$.
\item
There exist $x\in N_C(Q)$, $y\in N_Q(x)$, and a nonempty set
$A\subseteq N_Q(y)\setminus N_Q(x)$ such that, with
$G_0=G[y\to x;A]$, one of the following holds:
\begin{enumerate}[label=\textup{(\alph*)},nosep]
\item
$G_0$ is $2$-connected. In this case, set $G'=G_0$.
\item
$G_0$ is not $2$-connected, and there exists
$x'\in N_C(Q)\setminus\{x\}$ such that $G_0+yx'$ is $2$-connected.
In this case, set $G'=G_0+yx'$.
\end{enumerate}
\end{enumerate}
If alternative~\textup{(ii)} occurs, then $C$ remains locally maximal
in $G'$, $G'[V(C)]=G[V(C)]$, and $e(G')\ge e(G)$.
\end{lemma}

\begin{lemma}[Erd\H{o}s--Gallai~\cite{ErdosGallai1959}]
\label{lem:EG-path}
Let $G$ be a graph on $n$ vertices whose longest path has length at most
$\ell$. Then $e(G)\le {\ell n}/{2}$.
\end{lemma}

We now prove Lemma~\ref{lem:exterior-edge-bound}.
The framework used in the proof is adapted from
the proof of Lemma~4.3 in Ma and Ning~\cite{MaNing2020}. More precisely,
following their argument, we repeatedly apply the edge-switching lemma
of Fan, Lv, and Wang~\cite[Lemma~2.4]{FanLvWang2004} to obtain a
terminal graph in which, for every component $Q$ outside $C$, each
vertex of $N_C(Q)$ is adjacent to every vertex of $Q$. We then combine
the Erd\H{o}s--Gallai path bound with the local maximality of $C$.

\begin{proof}[\bf Proof of Lemma~\ref{lem:exterior-edge-bound}]
We first prove the following claim about repeated edge switching:

\begin{claim}\label{claim:switching-terminal}
Let $F$ be a $2$-connected graph in which $C$ is a locally maximal cycle of length $c<n$. 
We apply Lemma~\ref{lem:Fan-Lv-Wang-switch}(ii) repeatedly, starting from $F$, to obtain a graph $\widetilde F$. This graph has the following properties:
\begin{enumerate}[label=\textup{(\alph*)},nosep]
\item $e(\widetilde F)\ge e(F)$, $\widetilde F[V(C)]=F[V(C)]$, and $C$ is locally maximal in $\widetilde F$;
\item every component $Q$ of $\widetilde F-C$ satisfies $N_Q(x)=V(Q)$ for every $x\in N_C(Q)$;
\item let $Q_1,\dots,Q_m$ be the components of $\widetilde F-C$,
let $q_i=\lvert N_C(Q_i)\rvert$ and let $\ell_i$ be the maximum length of a path in $Q_i$.
Choose $Q=Q_j$ so that
$\ell_j+2q_j=\max_i(\ell_i+2q_i)$. Let
$\ell=\ell_j$ and $q=q_j$.
Then we have
$e(\widetilde F)\le e(\widetilde F[V(C)])+\frac{\ell+2q}{2}(n-c)$.
\end{enumerate}
\end{claim}

\begin{proof}
Let $F_0=F$, and for $j\ge 1$, let $F_j$ be the graph obtained from $F_{j-1}$ by one application of Lemma~\ref{lem:Fan-Lv-Wang-switch}(ii). If every component
of $F_j-C$ satisfies alternative \textup{(i)} of
Lemma~\ref{lem:Fan-Lv-Wang-switch}, then we stop. Otherwise, choose a
component $Q$ of $F_j-C$ for which alternative \textup{(i)} fails.
By Lemma~\ref{lem:Fan-Lv-Wang-switch}, alternative \textup{(ii)} applies.
Denote the resulting graph by $F_{j+1}$. Since the switched set is
nonempty, $e(F_{j+1}-C)<e(F_j-C)$. Hence the process terminates. Let
$\widetilde F$ be the terminal graph. Then
$e(\widetilde F)\ge e(F),\; \widetilde F[V(C)]=F[V(C)]$,
$C$ is locally maximal in $\widetilde F$, and
every component $Q$ of $\widetilde F-C$ satisfies
$N_Q(x)=V(Q)$ for every $x\in N_C(Q)$.
Thus \textup{(a)} and \textup{(b)} hold.

For each component $Q_i$ of $\widetilde F-C$, let $q_i = \lvert N_C(Q_i)\rvert$ and let $\ell_i$ be the maximum length of a path in $Q_i$.
By (b), every vertex of $N_C(Q_i)$ is adjacent to all vertices of $Q_i$. Hence
$e_{\widetilde F}(V(Q_i),V(C))=q_i\lvert Q_i\rvert$.
Since the subgraph $Q_i$ has no path longer than $\ell_i$, Lemma~\ref{lem:EG-path} gives
$e(\widetilde F[Q_i])\le \frac{\ell_i\lvert Q_i\rvert}{2}$.
Thus, 
$e(\widetilde F[Q_i])+e_{\widetilde F}(V(Q_i),V(C))
\le \left(\frac{\ell_i}{2}+q_i\right)\lvert Q_i\rvert
=\frac{\ell_i+2q_i}{2}\lvert Q_i\rvert$.
Now 
choose a component $Q_j$ maximizing $\ell_j+2q_j$, and let $\ell=\ell_j,\; q=q_j$.
Summing over all components, we obtain
$$
\sum_i \bigl(e(\widetilde F[Q_i])+e_{\widetilde F}(V(Q_i),V(C))\bigr)
\le \sum_i\frac{\ell_i+2q_i}{2}\lvert Q_i\rvert
\le \frac{\ell+2q}{2}\sum_i\lvert Q_i\rvert
=\frac{\ell+2q}{2}(n-c).
$$
Since the graph $\widetilde F$ consists of the subgraph inside $C$, the components outside $C$, and the edges joining them to $C$, we have
$$
e(\widetilde F)\le e(\widetilde F[V(C)])+\frac{\ell+2q}{2}(n-c).
$$
Thus (c) holds.
\end{proof}

\medskip
\textup{(i)}
Let
$\widehat G=\cl_C(G)$.
Since adding edges preserves 2-connectivity,
$\widehat G$ is $2$-connected.
By Lemma~\ref{lem:closure-preserves-local-maximality},
 $C$ is locally maximal in
$\widehat G$. 
Then $e(G)-e(G[V(C)])=e(\widehat G)-e(\widehat G[V(C)])$.
Apply Claim~\ref{claim:switching-terminal} with
$F=\widehat G$, and let $\widetilde F$ be the resulting terminal graph.
Then,
$e(\widehat G)\le e(\widetilde F)$ and $\widetilde F[V(C)]=\widehat G[V(C)]$.
Let $Q$ be the component given in Claim~\ref{claim:switching-terminal},
let $q=\lvert N_C(Q)\rvert$,
and let $\ell$ be the maximum length of a path in $Q$.
Thus, $e(\widehat G)-e(\widehat G[V(C)])\le e(\widetilde F)-e(\widetilde F[V(C)])\le \frac{\ell+2q}{2}(n-c)$.
It remains to show
$\ell+2q\le 2s$. 

First, suppose that $\ell=0$. Then $Q$ consists of a single vertex, say
$Q=\{z\}$. Fix an orientation of $C$ and set
$U=\{x^+:x\in N_C(Q)\}$.
Since the successor map on $C$ is injective,
$\lvert U\rvert=q$.
We claim that $U$ is an independent set in $\widetilde F[V(C)]$.

Suppose to the contrary that $x^+,y^+\in U$ are distinct and
$x^+y^+\in E(\widetilde F)$
for distinct $x,y\in N_C(Q)$. If $x^+=y$ or $y^+=x$, then replacing $xy$ of $C$ by $xzy$ gives a cycle of length
$c+1$, contradicting the local maximality of $C$. Otherwise, deleting
 $xx^+$ and $yy^+$ from $C$ and adding
$xz,\; zy,\; x^+y^+$
gives a cycle of length $c+1$, again contradicting the local
maximality of $C$. This proves that $U$ is independent.

By hypothesis, $\widehat G[V(C)]-S$
is a clique and $\lvert S\rvert=s-1$. Consequently, every independent set in
$\widetilde F[V(C)]$ contains at most one vertex outside $S$.
Hence every independent set
has size at most
$\lvert S\rvert+1=s$.
Therefore
$q=\lvert U\rvert\le s$.
Thus,
$\ell+2q=2q\le2s$.

Now suppose that $\ell\ge1$. 
Suppose to the contrary that $\ell +2q\ge 2s$.
Let
$P=uPv$
be a longest path of length $\ell$.
Since $\widetilde F$ is $2$-connected, $Q$ has at least two distinct
neighbors on $C$.
Hence $q\ge2$.
By Claim~\ref{claim:switching-terminal}\textup{(b)}, every vertex of
$N_C(Q)$ is adjacent to every vertex of $Q$. Thus, for any two distinct
vertices $x,x'\in N_C(Q)$,
$xuPv x'$
is an $(x,Q,x')$-path of length $\ell+2$. Since $x\ne x'$ and $u\ne v$,
$xu$ and $x'v$ are vertex-disjoint. Hence
$N_C(Q)$ is a strong attachment of $Q$ to $C$.
By Lemma~\ref{lem:Ma-Ning-strong-attachment},
$\omega(\widetilde F[V(C)])\le c-(\ell+1)(q-1)$.
Since $(\ell+1)(q-1)-(\ell+2q-3)=(\ell-1)(q-2)\ge0$, we have
$(\ell+1)(q-1)\ge\ell+2q-3$. If $s=2$, then
$(\ell+1)(q-1)\ge2=s$. If $s\ge3$, then
$(\ell+1)(q-1)\ge2s-3\ge s$.
It follows that
$\omega(\widetilde F[V(C)])\le c-s$.
By hypothesis, $\widehat G[V(C)]-S$
is a clique of order $c-s+1$.
Thus, $\omega(\widetilde F[V(C)])\ge c-s+1$,
contradicting $\omega(\widetilde F[V(C)])\le c-s$.
Hence
$\ell+2q\le2s$.
This proves \textup{(i)}.

\medskip
\textup{(ii)}
Apply Claim~\ref{claim:switching-terminal} with $F=G$, and let
$\widetilde G$ be the resulting terminal graph. Then
$e(\widetilde G)\ge e(G),\; \widetilde G[V(C)]=G[V(C)]$.
It follows that
$e(G-C)+e_G(G-C,C)
\le
e(\widetilde G-C)+e_{\widetilde G}(\widetilde G-C,C)$.

Let $Q$ be the component given in
Claim~\ref{claim:switching-terminal}, let $q=\lvert N_C(Q)\rvert$,
and let $\ell$ be the maximum length of a path in $Q$. Since
$\widetilde G$ is $2$-connected,
$q\ge2$.
Set
$N_C(Q)=\{x_1,\dots,x_q\}$
in cyclic order on $C$, with indices taken modulo $q$.
Let
$C_i=x_iCx_{i+1}$ and let $P$ be a longest path in $Q$. By Claim~\ref{claim:switching-terminal}\textup{(b)},
every vertex of $N_C(Q)$ is adjacent to every vertex of $Q$. If
$\ell=0$, the unique vertex of $P$ gives an
$(x_i,Q,x_{i+1})$-path of length $2$. If $\ell\ge1$, adding $x_i$
and $x_{i+1}$ to the two ends of $P$ gives such a path of length
$\ell+2$. Thus, in either case, there is an
$(x_i,Q,x_{i+1})$-path of length $\ell+2$.
If
$\lvert E(C_i)\rvert<\ell+2$
for some $i$, replacing $C_i$ by
the $(x_i,Q,x_{i+1})$-path of length $\ell +2$ will
give a cycle longer than $C$. The new cycle uses exactly two edges
between $C$ and $\widetilde G-C$, contradicting the local maximality of
$C$. Therefore
$\lvert E(C_i)\rvert\ge\ell+2$ for every $i$.
Summing over $i$ yields
$(\ell+2)q\le c$.
Thus,
$\ell+2q\le2\left\lfloor\frac c2\right\rfloor$.
By Claim~\ref{claim:switching-terminal}\textup{(c)},
$$
e(\widetilde G-C)+e_{\widetilde G}(V(\widetilde G)\setminus V(C),V(C))
= e(\widetilde G)-e(\widetilde G[V(C)])
\le\frac{\ell+2q}{2}(n-c)
\le\left\lfloor\frac c2\right\rfloor(n-c).
$$
This proves \textup{(ii)}.
\renewcommand{\qedsymbol}{$\blacksquare$}
\end{proof}

\section{Proof of Theorem~\ref{thm:main}}\label{sec:proof}

It suffices to prove the following two
lemmas.

\begin{lemma}\label{lem:base}
Let $k\ge2$ and $k+2\le n\le \ceil{{(5k+1)}/{2}}$. Then every
$n$-vertex graph containing no $(k+2)$-path-fan has at most $F_k(n)$ edges.
\end{lemma}

\begin{lemma}\label{lem:terminal}
Let $k\ge2$ and $n\ge \ceil{{(5k+1)}/{2}}$.
Then every $n$-vertex graph containing no $(k+2)$-path-fan has at most $(k+1)(n-k-1)$ edges.
\end{lemma}

We first prove Lemma~\ref{lem:base}. To prove it, we prove two helpful lemmas, i.e., Lemmas \ref{lem:cycle-degree} and \ref{lem:minimal-counterexample-setup}.

\begin{lemma}\label{lem:cycle-degree}
Let $k\ge1$, let $G$ be a graph containing no $(k+2)$-path-fan, and let $C$ be any cycle of $G$. Then $d_{G[V(C)]}(v)\le k+1$ for each $v\in V(C)$.
\end{lemma}

\begin{proof}
If $d_{G[V(C)]}(v)\ge k+2$ for some $v\in V(C)$, then $C-v$
is a path in $G-v$ containing at least $k+2$ neighbors of $v$.
Thus $v$ and $C-v$ form a $(k+2)$-path-fan, a contradiction.
\renewcommand{\qedsymbol}{$\blacksquare$}
\end{proof}

\begin{lemma}\label{lem:minimal-counterexample-setup}
Fix $k\ge2$.
Suppose that Lemma~\ref{lem:base} is false. Choose a counterexample $G$
with $n:=\lvert V(G)\rvert$ minimum and, subject to this, with $e(G)$ minimum.
Then the following hold:

\begin{itemize}[nosep]
    \item[\upshape(1)] $e(G) = F_k(n) + 1$;
    \item[\upshape(2)] $G$ is non-Hamiltonian;
    \item[\upshape(3)] $G$ is $2$-connected;
    \item[\upshape(4)] $\delta(G) \ge \lfloor {(k+1)}/{2} \rfloor + 1$;
    \item[\upshape(5)] 
    If $C$ is a longest cycle of $G$ with length $c$, then
    $2\lfloor {(k+1)}/{2} \rfloor + 2 \le c < n$.
\end{itemize}
\end{lemma}

\begin{proof}
(1) Since deleting edges cannot create a $(k+2)$-path-fan, the choice of $G$
gives $e(G)=F_k(n)+1$.

(2) Since $F_k(n)\ge R_k(n)$, we have
$e(G)>(k+1)n/2$. Thus some vertex $x$ has degree at least $k+2$. If
$C$ is a Hamilton cycle, then $C-x$ and $x$ form a
$(k+2)$-path-fan. Hence $G$ is non-Hamiltonian.

For (3) and (4), we use the following two properties.

\begin{claim}\label{claim:increment-properties}
\begin{enumerate}[label=\textup{(\roman*)}, nosep]

\item For every integer $m\ge 1$ and every $a\in I_k$ such that $J_k(m)=a(m - a) + \left\lfloor \frac{a(k + 1 - a)}{2}\right\rfloor$,
 we have $J_k(m+1)\ge J_k(m)+a$.

\item If $F_k(n)=R_k(n)$, then $F_k(m)=R_k(m)$ for every $k+2\le m<n$.
\end{enumerate}
\end{claim}

\begin{proof}
(i) By the definition of $J_k(m+1)$,
$J_k(m+1)\ge a(m+1-a)+\left\lfloor\frac{a(k+1-a)}{2}\right\rfloor=J_k(m)+a$.

(ii) For $k+2\le m<n$,
let $a\in I_k$ such that $J_k(m)=a(m - a) + \left\lfloor \frac{a(k + 1 - a)}{2}\right\rfloor$.
By (i), we have $J_k(n)\ge a(n-a)+\lfloor a(k+1-a)/2\rfloor=J_k(m)+(n-m)a$. Since $F_k(n)=R_k(n)$, we have $J_k(n)\le R_k(n)$.
Since $a\ge\lfloor(k+1)/2\rfloor+1>(k+1)/2$, we have $(n-m)a+\lfloor (k+1)m/2\rfloor\ge\lfloor (k+1)n/2\rfloor=R_k(n)$.
Then
$J_k(m)-R_k(m)\le J_k(n)-(n-m)a-R_k(m)=J_k(n)-((n-m)a+R_k(m))\le J_k(n)-R_k(n)\le0$.
Hence $J_k(m)\le R_k(m)$, which implies $F_k(m)=R_k(m)$.
This proves (ii).
\end{proof}

(3) First, suppose that $G$ is disconnected, and let
$Q_1,\ldots,Q_t$ be its components, where $q_i:=\lvert V(Q_i)\rvert$. 
 If $q_i\le k+1$, then
$e(Q_i)\le\binom{q_i}{2}=F_k(q_i)$.
If $q_i\ge k+2$, then the minimality of $n$ gives
$e(Q_i)\le F_k(q_i)$.
Thus, $e(Q_i)\le F_k(q_i)$ for every $i$.
If $J_k(n)>R_k(n)$, repeated application of
Lemma~\ref{lem:merge-ineq}\textup{(ii)} gives
$e(G)\le\sum_i F_k(q_i)\le F_k(n)$, contradicting (1). If
$J_k(n)\le R_k(n)$, Claim~\ref{claim:increment-properties}\textup{(ii)} and
the definition of $F_k$ give
$e(G)\le\sum_i F_k(q_i)
 \le\sum_i\left\lfloor\frac{(k+1)q_i}{2}\right\rfloor
 \le R_k(n)=F_k(n)$,
a contradiction. Hence $G$ is connected.

Suppose that $z$ is a cut-vertex. Let $Q_1,\ldots,Q_t$ be the
components of $G-z$, and set
$G_i:=G[V(Q_i)\cup\{z\}]$ and $n_i:=\lvert V(G_i)\rvert$.
If $n_i\le k+1$, then
$e(G_i)\le \binom{n_i}{2}= F_k(n_i)$.
If $n_i\ge k+2$, then the minimality of $n$ gives
$e(G_i)\le F_k(n_i)$.
Thus, $e(G)\le\sum_ie(G_i)\le \sum_iF_k(n_i)$. 
We consider the cases $J_k(n)>R_k(n)$ and $J_k(n)\le R_k(n)$.
If $J_k(n)>R_k(n)$,
then by Lemma~\ref{lem:merge-ineq}\textup{(i)}, 
$e(G)\le \sum_i F_k(n_i)\le F_k\big(\sum_in_i-(t-1)\big)=F_k(n)=J_k(n)$,
contradicting $(1)$.
If $R_k(n)\ge J_k(n)$, 
by Lemma~\ref{lem:linear-bound}(ii), $n\le 2k+2$. 
By Claim~\ref{claim:increment-properties}(ii), if $n_i\ge k+2$,
$F_k(n_i)=R_k(n_i)\le \floor{{(k+1)n_i}/{2}}$.
Define $I=\{i:n_i\ge k+2\}$. If $\lvert I\rvert\ge 2$, then $n=1+\sum_i (n_i-1)\ge 1+2(k+1)=2k+3$, contradicting $n\le 2k+2$. Hence $\lvert I\rvert\le 1$.
If $I=\varnothing$, then $e(G)\le\sum_i\binom{n_i}{2}\le\sum_i\lfloor\frac{(k+1)(n_i-1)}2\rfloor\le\lfloor\frac{(k+1)(n-1)}2\rfloor<R_k(n)$.
If  $I=\{i_0\}$, then $e(G)\le R_k(n_{i_0})+\sum_{i\ne i_0}\binom{n_i}{2}\le\lfloor\frac{(k+1)n_{i_0}}2\rfloor+\sum_{i\ne i_0}\lfloor\frac{(k+1)(n_i-1)}2\rfloor\le R_k(n)$,
contradicting $(1)$.
Hence $G$ has no cut-vertex.
This proves (3).

(4) Let $v\in V(G)$. If
$n-1\le k+1$, then $e(G-v)\le F_k(n-1)\le R_k(n-1)$.
Hence
$d_G(v)\ge R_k(n)+1-R_k(n-1)\ge \lfloor(k+1)/2\rfloor+1$. We may therefore assume that
$n-1\ge k+2$. By the minimality of $n$,
$d_G(v)=e(G)-e(G-v)\ge F_k(n)+1-F_k(n-1)$.
If $J_k(n)\le R_k(n)$, then Claim
\ref{claim:increment-properties}\textup{(ii)} gives
$F_k(n-1)=R_k(n-1)$.
Thus, $d_G(v)\ge R_k(n)+1-R_k(n-1)\ge \lfloor(k+1)/2\rfloor+1$.
Suppose
 that $J_k(n)>R_k(n)$. If $F_k(n-1)=J_k(n-1)$, then Claim
\ref{claim:increment-properties}\textup{(i)} gives
$F_k(n)-F_k(n-1)\ge \lfloor(k+1)/2\rfloor+1$. If $F_k(n-1)=R_k(n-1)$, then
$F_k(n)\ge R_k(n)+1$.
Therefore $d_G(v)\ge \lfloor(k+1)/2\rfloor+1$. 
This proves (4).

(5) Let $C$ be a longest cycle of $G$ and let $c=\lvert E(C)\rvert$. By
\textup{(2)} and \textup{(3)}, the graph $G$ is non-Hamiltonian and
$2$-connected. Hence Theorem~\ref{thm:Dirac-cycle}  gives
$n>c\ge 2\delta(G)
\ge2\floor{{(k+1)}/{2}}+2$.

This completes the proof.
\renewcommand{\qedsymbol}{$\blacksquare$}
\end{proof}

Now we are in a position to prove Lemma~\ref{lem:base}. 

\begin{proof}[\bf Proof of Lemma~\ref{lem:base}]
Suppose that Lemma~\ref{lem:base} is false, and let $G$ be the minimal counterexample given in Lemma~\ref{lem:minimal-counterexample-setup}. Let $C$ be a longest cycle of $G$. Set
$H=G[V(C)]$ and
$c=\lvert E(C)\rvert$. 
Let
$\widehat G=\cl_C(G)$, 
$\widehat H=\widehat G[V(C)]$.
Since $G$ is $2$-connected and non-Hamiltonian,
Lemma~\ref{lem:Ma-Ning-cycle-dichotomy} applies.
If Lemma~\ref{lem:Ma-Ning-cycle-dichotomy} (i) holds,
there exist an integer $s$ and a set $S\subseteq V(C)$ such that
$2\le s\le \bigl\lfloor\frac{c}{2}\bigr\rfloor-1,\lvert S\rvert=s-1$,
$d_{\widehat H}(u)\le s$ for every $u\in S$, and $\widehat H-S$ is a clique.
By Lemma~\ref{lem:exterior-edge-bound}(i),
$e(G)\le e(H)+s(n-c)$.
Since $H\subseteq \widehat H$ and $d_H(v)\le k+1$ for every $v\in V(C)$,
$2e(H)\le (s-1)s+(c-s+1)(k+1)$.
Therefore
$e(G)\le
s(n-c)+
\Bigl\lfloor\frac{(s-1)s+(c-s+1)(k+1)}{2}\Bigr\rfloor
\le F_k(n)$
by Lemma~\ref{lem:cycle-arithmetic-consequences}\textup{(i)}, contradicting
Lemma~\ref{lem:minimal-counterexample-setup}\textup{(1)}.
We may therefore assume that Lemma~\ref{lem:Ma-Ning-cycle-dichotomy} (ii) holds.
Then there is a set $R\subseteq V(C)$ such that $\lvert R\rvert=\bigl\lfloor\frac{c}{2}\bigr\rfloor-1$ and
$d_{\widehat H}(u)\le \bigl\lfloor\frac{c}{2}\bigr\rfloor$ for every $u\in R$.
Since Theorem~\ref{thm:MN-exterior-stability} applies when $c\ge10$, we first distinguish the cases $c\le9$ and $c\ge10$.
For convenience, set $q=\floor{c/2}$.

\medskip
\noindent\textbf{Case 1.}
$c\le 9$.

Since $c\ge 2\lfloor(k+1)/2\rfloor+2$
by Lemma~\ref{lem:minimal-counterexample-setup}\textup{(5)},
we have
$2\le k\le 6$.
Since $q\le k+1$, we have
$2e(H)\le(q-1)q+(c-q+1)(k+1)$.
By Lemma~\ref{lem:exterior-edge-bound}(ii),
$e(G-C)+e_G(G-C,C)
 \le q(n-c)$.
It follows that
$$
e(G)\le q(n-c)+
\left\lfloor\frac{(q-1)q+(c-q+1)(k+1)}2\right\rfloor.
$$
If $(k,n,c)\neq(3,8,6)$, then
Lemma~\ref{lem:small-c-arithmetic} gives $e(G)\le F_k(n)$, a
contradiction. 
It remains to consider $(k,n,c)=(3,8,6)$. Then
$e(G)=F_3(8)+1=17$. Moreover, $\lvert R\rvert=2$,
$d_{\widehat H}(u)\le3$ for $u\in R$, and Lemma~\ref{lem:cycle-degree} gives
$d_H(x)\le4$ for $x\in V(C)$. Hence
$2e(H)\le2\cdot3+4\cdot4=22$ and $e(H)\le11$.
By Lemma~\ref{lem:exterior-edge-bound}(ii),
$e(G-C)+e_G(G-C,C)\le3(8-6)=6$.
Consequently,
\begin{equation}\label{eq:k3-equalities}
e(G)=17,
\qquad
e(H)=11,
\qquad
e(G-C)+e_G(G-C,C)=6.
\end{equation}

By Lemma~\ref{lem:minimal-counterexample-setup}\textup{(4)},
$\delta(G)\ge3$.
Let $V(G)\setminus V(C)=\{u,v\}$.
We claim that $uv\notin E(G)$. Otherwise,
 we have
$\lvert N_C(u)\rvert,\lvert N_C(v)\rvert\ge2$ and \eqref{eq:k3-equalities} gives
$\lvert N_C(u)\rvert+\lvert N_C(v)\rvert=5$. 
We may assume that $\lvert N_C(u)\rvert=3$.
For each $y\in N_C(v)$, there is an $x\in N_C(u)\setminus\{y\}$ at
distance at most two from $y$ on $C$.
Since $xuvy$ is a path of length 3,
replacing the shorter segment $xCy$ by $xuvy$ gives a cycle of length at
least $7$, a contradiction.
Thus, $uv\notin E(G)$ and 
$\lvert N_C(u)\rvert+\lvert N_C(v)\rvert=6$. Since $\delta(G)\ge3$, it follows that
$\lvert N_C(u)\rvert=\lvert N_C(v)\rvert=3$.
Let $C=a_1b_1a_2b_2a_3b_3a_1$.
Since $C$ is a longest cycle of $G$,
neither $N_C(u)$ nor $N_C(v)$ contains two consecutive vertices of $C$.
Hence $N_C(u)$ and
$N_C(v)$ are the two independent sets of order $3$ in $C$.
If
$N_C(u)=\{a_1,a_2,a_3\}$
and 
$N_C(v)=\{b_1,b_2,b_3\}$,
then
$u a_1 b_1 a_2 b_2 v b_3 a_3 u$
is an $8$-cycle, a contradiction. Thus $N_C(u)=N_C(v)$. Without loss of
generality,
assume that $N_C(u)=N_C(v)=\{a_1,a_2,a_3\}$.
If, say, $b_1b_2\in E(G)$, then
$ua_1b_3 a_3 b_2 b_1 a_2 u$
is a $7$-cycle, a contradiction.

By \eqref{eq:k3-equalities}, the cycle $C$ has five chords. 
Since no chord
joins two vertices of $\{b_1,b_2,b_3\}$, every chord has one of the
following types:
\begin{itemize}[nosep]
  \item Type I: an edge inside $\{a_1,a_2,a_3\}$ (at most $3$);
  \item Type II:  $a_1b_2$, $a_2b_3$ and $a_3b_1$.
\end{itemize}
Let $m$ be the number of type-I chords. Since there are at most three
type-II chords, $m\ge2$.

In $H$, every $a_i$ has degree at most $4$ and is already incident with
two edges of the cycle $C$.
Hence, each $a_i$ is incident with at most two chords.
Therefore, there are at
most six chord incidences at $a_1,a_2,a_3$. On the other hand, each type-I chord
contributes two such incidences, while each type-II chord contributes
one. Since there are $m$ type-I chords and $5-m$ type-II chords, the
total number of chord incidences is
$2m+(5-m)=m+5\ge 7$, a contradiction.

\medskip
\noindent\textbf{Case 2.}
$c\ge 10$.

We distinguish according to whether
$e(G)-e(H)\le (q-1)(n-c)$. In the second
subcase, we use Theorem~\ref{thm:MN-exterior-stability}.

\medskip
\noindent\textbf{Subcase 1.}
$e(G)-e(H)\le (q-1)(n-c)$.

By Lemma~\ref{lem:cycle-degree},
$d_H(v)\le k+1$ for every  $v\in V(C)$.
Hence
$$\begin{aligned}
    2e(H)&=\sum_{v\in V(C)}d_H(v)
    =\sum_{u\in R}d_H(u)+\sum_{v\in V(C)\setminus R}d_H(v)\\
    &\le
(q-1)\min\{q,k+1\}+(c-q+1)(k+1).
\end{aligned}
$$
By Lemma~\ref{lem:cycle-arithmetic-consequences}\textup{(ii)}, we have
$$
\begin{aligned}
e(G)&\le(q-1)(n-c)+
\Bigl\lfloor
\frac{
(q-1)\min\{q,k+1\}+(c-q+1)(k+1)
}{2}
\Bigr\rfloor\le F_k(n),
\end{aligned}
$$
 contradicting
Lemma~\ref{lem:minimal-counterexample-setup}\textup{(1)}.

\medskip
\noindent\textbf{Subcase 2.}
$e(G)-e(H)>(q-1)(n-c)$.

We begin with analyzing the structure of $G$.

\begin{claim}\label{claim:ext-rich-host1}
$G\subseteq W_{n,q,c}$.
\end{claim}
\begin{proof}
Suppose not.
By Lemma~\ref{lem:minimal-counterexample-setup}, $G$ is $2$-connected and
$c<n$. Since $c\ge10$, Theorem~\ref{thm:MN-exterior-stability} applies.
Then $c$ is odd and $G\subseteq F$ for some
$F\in\mathcal X_{n,c}\cup\mathcal Y_{n,c}$.
If $k\le3$, then $n\le8$, contradicting $c\ge10$. If
$k=4$, then $n\le11$ and $10\le c<n$. 
Then $c=10$, contradicting the
fact that $c$ is odd. Thus $k\ge5$. 
Then $\delta(G)\ge\lfloor{(k+1)}/{2}\rfloor+1\ge4$.
By
Lemma~\ref{lem:exceptional-low-degree-or-W}, either
$F\subseteq W_{n,q,c}$ or $F$
has a vertex of degree at most $3$. 
If $F\subseteq W_{n,q,c}$, then $G\subseteq W_{n,q,c}$.
If $\delta(F)\le 3$, then $\delta (G)\le 3$,
 a contradiction.
Hence $G\subseteq W_{n,q,c}$.
\end{proof}

Since $G\subseteq W_{n,q,c}$,
$V(G)$ can be partitioned into three disjoint vertex sets $S,T,U$ with
$\lvert S\rvert=q$ and $\lvert T\rvert=\tau:=c-2q+1\in\{1,2\}$.
\begin{claim}\label{claim:ext-rich-host2}
$S\subseteq V(C)$ and $\lvert T\setminus V(C)\rvert\le1$.
\end{claim}

\begin{proof}
Orient $C$ and let $u^+$ be the successor of $u$ on $C$. If
$u\in U\cap V(C)$, then $u^+\in S\cap V(C)$ since the vertices in $U$ are
adjacent only to the vertices in $S$. Since the successor map is injective,
$\lvert U\cap V(C)\rvert\le \lvert S\cap V(C)\rvert$. Hence, we have
$ c\le 2\lvert S\cap V(C)\rvert+\lvert T\cap V(C)\rvert\le 2\lvert S\cap V(C)\rvert+\tau $.
If $S\not\subseteq V(C)$, then $\lvert S\cap V(C)\rvert\le q-1$.
Therefore
$ c\le 2(q-1)+\tau=c-1 $, a contradiction. Thus $S\subseteq V(C)$.
Now
$ c\le 2q+\lvert T\cap V(C)\rvert $,
while $c=2q+\tau-1$. It follows that
$\lvert T\cap V(C)\rvert\ge\tau-1$. Hence $\lvert T\setminus V(C)\rvert\le1$.
\end{proof}

By Claim~\ref{claim:ext-rich-host2}, we have $S\subseteq V(C)$ and
$\lvert T\setminus V(C)\rvert\le 1$. Since $U$ is independent, $E_G(U,T)=\emptyset$ and $\lvert T\setminus V(C)\rvert \le 1$, the graph $G-C$ has no edges.
Thus $e(G)-e(H)=e_G(G-C,C)$.

We first prove that
$q\in I_k$.
Since $e(G)-e(H)\ge (q-1)(n-c)+1$, 
there exists a vertex $z\in
V(G)\setminus V(C)$ with at least $q$ neighbors on $C$.
If $q\ge k+2$, 
then there exists a subpath $P$ on $C$ that contains at least $k+2$ neighbors of $z$.
Thus, $z$ and $P$ form a $(k+2)$-path-fan,
a contradiction. Therefore
$q\le k+1$.
On the other hand, Lemma~\ref{lem:minimal-counterexample-setup} gives
$c\ge 2\bigl\lfloor(k+1)/2\bigr\rfloor+2$.
Then,
$q\ge\bigl\lfloor(k+1)/2\bigr\rfloor+1$. Hence
$q\in I_k$.

Now
$$
e(G)-e(H)=
\sum_{v\in U\setminus V(C)}\lvert N_G(v)\cap V(C)\rvert
+\sum_{v\in T\setminus V(C)}\lvert N_G(v)\cap V(C)\rvert.
$$
Set $\tau_C=\lvert T\cap V(C)\rvert$. Then
$\lvert T\setminus V(C)\rvert=\tau-\tau_C$ and $\tau_C\in\{\tau-1,\tau\}$.
If $v\in U\setminus V(C)$, then all its neighbors are in $S$.
Thus, $\lvert N_G(v)\cap V(C)\rvert\le q$. If $v\in T\setminus V(C)$, then its neighbors on $C$ are
in $S\cup (T\cap V(C))$.
Hence $\lvert N_G(v)\cap V(C)\rvert\le q+\tau_C$. Since
$\lvert U\setminus V(C)\rvert=(n-c)-(\tau-\tau_C)$,
we obtain
\begin{equation}\label{eq:outside-edge-bound}
\begin{aligned}
e(G)-e(H)
&\le q\bigl(n-c-(\tau-\tau_C)\bigr) + (q+\tau_C)(\tau-\tau_C) = q(n-c) + \tau_C(\tau-\tau_C).
\end{aligned}
\end{equation}
Let $X=V(C)\setminus S=(T\cup U)\cap V(C)$.
Then $\lvert X\rvert=c-q$.
Since $U$ is independent and has no neighbors in $T$, we have
\begin{equation}\label{eq:cycle-edge-bound}
e(H)\le e(G[S])+e_G\bigl(S,X\bigr)+\binom{\tau_C}{2}.
\end{equation}
Since $\tau\in\{1,2\}$ and $\tau_C\in\{\tau-1,\tau\}$, we have
$\binom{\tau_C}{2}+\tau_C(\tau-\tau_C)=\tau-1\in\{0,1\}$.
By \eqref{eq:outside-edge-bound} and \eqref{eq:cycle-edge-bound},
\begin{equation}\label{eq:e(G)}
\begin{aligned}
e(G)
&\le e_G(S,X)+e(G[S])+\binom{\tau_C}{2}+q(n-c)+ \tau_C(\tau-\tau_C)\\
&\le e_G(S,X)+e(G[S])+q(n-c)+ \tau-1.
\end{aligned}
\end{equation}
Applying Lemma~\ref{lem:cycle-degree} to the vertices of $S\subseteq
V(C)$ gives
\begin{equation}\label{eq:outside-edge-bound3}
2 e(G[S])+e_G\bigl(S,X\bigr)=
\sum_{s\in S} d_H(s)\le q(k+1).
\end{equation}

It remains to prove that $e(G)\le J_k(n)$.
First, suppose that $\lvert X\rvert=c-q\le k+1$.
Thus, by \eqref{eq:outside-edge-bound3},
$e(G[S])\le \bigl\lfloor \frac{q(k+1)-e_G(S,X)}{2}\bigr\rfloor$.
Therefore,
$$
e(G)
\le e_G(S,X)+\Bigl\lfloor \frac{q(k+1)-e_G(S,X)}{2}\Bigr\rfloor
   +q(n-c)+ \tau-1.
$$
Since $e_G(S,X)\le q(c-q)$,
$$
\begin{aligned}
e(G) &\le q(c-q)+\left\lfloor \frac{q(k+1)-q(c-q)}{2} \right\rfloor + q(n-c)+\tau-1\\
&= q(n-q)+\left\lfloor \frac{q(k+1-c+q)}{2} \right\rfloor+\tau-1.
\end{aligned}
$$
Since $c=2q+\tau-1$, $k+1-c+q = k+2-q-\tau$. Hence
$$
e(G)\le q(n-q)+\left\lfloor \frac{q(k+2-q-\tau)}{2} \right\rfloor+\tau-1.
$$
If $\tau=1$, then $e(G)\le q(n-q)+\left\lfloor \frac{q(k+1-q)}{2}\right\rfloor \le J_k(n)$.
If $\tau=2$, then $c=2q+1$. 
Since $c-q\le k+1$, $q\le k$. 
Then,
$$
\begin{aligned}
e(G) &\le q(n-q)+\left\lfloor \frac{q(k-q)}{2}\right\rfloor +1= q(n-q)+\left\lfloor \frac{q(k-q)}{2}+1\right\rfloor\\
&\le q(n-q)+\left\lfloor \frac{q(k+1-q)}{2}\right\rfloor \le J_k(n).
\end{aligned}
$$

Next, we assume that $c-q>k+1$. Since
$c=2q+\tau-1$ and $\tau\in\{1,2\}$, if $\tau=1$, then
$c-q=q\le k+1$, a contradiction. Hence $\tau=2$. Then,
$c-q=q+1$ and $q+1>k+1$, that is, $q>k$. Since
$q\le k+1$, $q=k+1$. From \eqref{eq:outside-edge-bound3},
$2 e(G[S])+e_G\bigl(S,X\bigr)=
\sum_{s\in S} d_H(s)\le q(k+1)=q^2$.
Hence
\begin{align*}
e(G)
&\le {q}^2+q(n-c)+1= {q}^2+q(n-2q-1)+1= q(n-q)-q+1\le q(n-q).
\end{align*}
Thus $e(G)\le J_k(n)$.

Therefore,
$e(G)\le J_k(n)\le F_k(n)$, contradicting
$e(G)=F_k(n)+1$.
\renewcommand{\qedsymbol}{$\blacksquare$}
\end{proof}

The proof of Lemma~\ref{lem:terminal} is very short. 

\begin{proof}[\bf Proof of Lemma~\ref{lem:terminal}]
We proceed by induction on $n$.
The base case  $n=\lceil(5k+1)/2\rceil$ follows from  Lemmas~\ref{lem:base} and~\ref{lem:B-equals-f}.

Let $n>\lceil(5k+1)/2\rceil$, and let $G$ be an $n$-vertex graph containing no $(k+2)$-path-fan. Let $P=v_0v_1\dots v_t$ be
a longest path. Then every neighbor of $v_0$ lies on
$P$. If $d_G(v_0)\ge k+2$, then $v_0$ and the path
$v_1v_2\dots v_t$ form a $(k+2)$-path-fan, a contradiction. Hence
$d_G(v_0)\le k+1$. By induction and Lemma~\ref{lem:B-equals-f},
$e(G)=e(G-v_0)+d_G(v_0)\le (k+1)(n-k-2)+(k+1)=(k+1)(n-k-1).
$
This completes the proof.
\renewcommand{\qedsymbol}{$\blacksquare$}
\end{proof}

Now, we can finish the proof of Theorem~\ref{thm:main}. 

\begin{proof}[{\bf Proof of Theorem~\ref{thm:main}}]
If $k=1$, then $g_1(3)=3$.
P\'osa's theorem and the
construction $K_{2,n-2}$ give $g_1(n)=2n-4$ for $n\ge4$.  We may therefore assume that $k\ge2$.
If $k+2\le n<\lceil(5k+1)/2\rceil$, then
Lemma~\ref{lem:base} gives $g_k(n)\le F_k(n)$. The two constructions in
Section~\ref{sec:constructions} give 
$g_k(n)\ge R_k(n)$ and $g_k(n)\ge J_k(n)$.
Thus, $g_k(n)=F_k(n)$.

Now let $n\ge \lceil(5k+1)/2\rceil$. Lemma~\ref{lem:terminal} gives
$g_k(n)\le (k+1)(n-k-1)=J_k(n)$.
Construction~\ref{cons:split-construction} with $a=k+1$ gives a graph containing no $(k+2)$-path-fan and having exactly
$(k+1)(n-k-1)$ edges.
Thus,
$g_k(n)=J_{k}(n)=(k+1)(n-k-1)$.
By Lemma~\ref{lem:B-equals-f}, $F_k(n)=J_k(n)$.
This completes the proof.
\renewcommand{\qedsymbol}{$\blacksquare$}
\end{proof}

\section*{Acknowledgements}
The authors thank Thomas Bloom for founding and maintaining the Erd\H{o}s Problems website, which has been a valuable resource and inspiration for this work.
This work was presented at the **2nd Meili Combinatorics Conference** held at Xinjiang Normal University. The second author thanks the conference organizers for the welcoming atmosphere.

\section*{Declaration of AI usage}
This work began while the authors were using GPT-5.5 Pro. Early on, the AI system could not provide useful mathematical assistance. The first author proved several small-$k$ ($k\le 7$), and the second author proved a stability version of Jiang’s theorem, relying heavily on techniques from his earlier joint paper with Ma \cite{MaNing2020}.

The second author then combined the two manuscripts and provided them to the AI, asking it to propose a general conjecture. Its first proposed conjecture was incorrect; after further exchanges, its second conjecture was given and became the main theorem of this paper. Thus, the AI contributed to the formulation of the general conjecture, based on mathematical results and proofs already obtained by the authors. The authors also used AI for some calculations in auxiliary lemmas and for language editing and proofreading.

The current version was completely written and checked by the authors, who take full responsibility for the mathematical content and the final text. All original mathematical results proved in this paper, including all lemmas and claims proved in Sections~\ref{sec:constructions}--\ref{sec:proof} as well as the main theorem, Theorem~\ref{thm:main}, have been formalized and machine-checked in Lean~4.
The complete Lean source code, the paper-to-Lean
correspondence table, the pinned Lean and Mathlib versions, and the
automated verification workflow are available at 
\url{https://github.com/xzchen-math/Erdos767lean}.

\end{document}